\documentclass[11pt]{article}
\usepackage{amssymb}
\usepackage{amsfonts}
\usepackage{amsmath}
\usepackage{mathrsfs}
\usepackage{graphicx}
\usepackage{amsbsy}
\usepackage{theorem}
\usepackage{color}
\usepackage{bm}
\usepackage{hyperref}
\usepackage{tikz}
\usepackage[normalem]{ulem}
\bf
\newcommand{\vr}{\rho}

\newcommand{\Div}{{\rm div}}
\newcommand{\Dt}{\frac{ {\rm d}}{\dt}}
\newcommand{\vu}{\vc{u}}

\newcommand{\vc}[1]{{\bf #1}}

\newcommand{\Grad}{\nabla}

\newcommand{\lr}[1]{\left( #1 \right)}
\newcommand{\ep}{\varepsilon}

\newcommand{\eq}[1]{\begin{equation}
\begin{split}
#1
\end{split}
\end{equation}}
\newcommand{\eqh}[1]{\begin{equation*}
\begin{split}
#1
\end{split}
\end{equation*}}

\newcommand{\intO}[1]{\int_{\RR^d} #1 \ \dx}
\newcommand{\iintO}[1]{\int_{\RR^{2d}} #1 \ \dx\, \dy}
\newcommand{\intRd}[1]{\int_{{\mathbb{R}}^d} #1 \ \dx}

\newcommand{\intRdB}[1]{\int_{{\mathbb{R}}^d} \left( #1 \right) \ \dx}
\newcommand{\intT}[1]{\int_0^T #1 \ \dt}
\newcommand{\intTO}[1]{\int_0^T\!\!\!\! \int_{\RR^d} #1 \ \dxdt}
\newcommand{\iintTO}[1]{\int_0^T\!\!\!\! \int_{\RR^{2d}} #1 \ \dx\, \dy\, \dt}
\newcommand{\intTOB}[1]{ \int_0^T\!\!\!\! \int_{\RR^d} \left( #1 \right) \ \dxdt}

\newcommand{\dx}{{\rm d} {x}}
\renewcommand{\dh}{{\rm d} {h}}
\newcommand{\dy}{{\rm d} {y}}
\newcommand{\dt}{{\rm d} t }
\newcommand{\dxdt}{\dx \, \dt}
\newcommand{\gamman}{{\gamma_n}}

\newtheorem{theorem}{Theorem}[section]
\newtheorem{lemma}[theorem]{Lemma}
\newtheorem{corollary}[theorem]{Corollary}
\newtheorem{definition}[theorem]{Definition}
\newtheorem{proposition}[theorem]{Proposition}
\newtheorem{remark}[theorem]{Remark}

\def\thetheorem{\thesection.\arabic{theorem}}
\def\thesection{\arabic{section}}

\def\beq{\begin{equation}\displaystyle}
\def\eeq{\end{equation}}
\def\bel{\begin{equation} \displaystyle \begin{array}{l} }
\def\eel{\end{array} \end{equation} }
\def\bell{\begin{equation} \displaystyle \begin{array}{ll}  }
\def\eell{\end{array} \end{equation} }

\def\bea{\begin{eqnarray}}
\def\eea{\end{eqnarray} }
\def\bean{\begin{eqnarray*}}
\def\eean{\end{eqnarray*} }
\newenvironment{proof}{\noindent{\bf Proof.~}}
{{\mbox{}\hfill {\small \fbox{}}\\}}
\catcode`@=11
\renewcommand\appendix{\bigskip {\noindent \Large \bf Appendix}
  \setcounter{section}{0}%
  \setcounter{subsection}{0}%
\setcounter{equation}{0}%
\setcounter{theorem}{0}%
\def\thetheorem{A.\arabic{theorem}}
\def\theequation {A.\arabic{equation}}}
\catcode`@=12

\newcommand{\dv}{\mathop{\rm div}\nolimits}

\def\NN{\mathbb{N}}

\def\RR{\mathbb{R}}

\def\ds{\displaystyle}
\def\bs{\bigskip}

\def\eps{\varepsilon}
\def\bar#1{{\overline #1}}

\def\pa{\partial}

\def\calE{{\mathcal E}}

\def\calM{{\mathcal M}}

\def\calR{{\mathcal R}}
\def\calD{{\mathcal D}}

\def\calJ{{\mathcal J}}
\def\calK{{\mathcal K}}

\begin{document}

\title{Incompressible limit of the Navier-Stokes model with a growth term}

\author{Nicolas Vauchelet\thanks{Universit\'e Paris 13, Sorbonne Paris Cit\'e, CNRS UMR 7539, 
Laboratoire Analyse G\'eom\'etrie et Applications,
93430 Villetaneuse, France.  Email: \texttt{vauchelet@math.univ-paris13.fr}}, 
Ewelina Zatorska\thanks{Imperial College London, Department of Mathematics,  London SW7 2AZ, United Kingdom. Email: \texttt{e.zatorska@imperial.ac.uk}}
\footnote{Part of this work was done while N.V. was a CNRS fellow at Imperial College London, 
he is very grateful to the CNRS and to Imperial College from its hospitality.
N.V. acknowledges partial support from the ANR blanche project
Kibord No ANR-13-BS01-0004 funded by the French Ministry of Research. E.Z. was supported by the the Department of Mathematics, Imperial College, through a Chapman Fellowship, and by the National Science Centre grant 2014/14/M/ST1/00108
(Harmonia).}
}

\maketitle

\begin{abstract}
Starting from isentropic compressible Navier-Stokes equations with growth term in the continuity equation, we rigorously justify that performing an incompressible limit one arrives to the two-phase free boundary fluid system.
\end{abstract}

\bs



\bs

\section{Introduction}
The purpose of this work is to analyze the Navier-Stokes equations that generalize the fluid-based models of tumors. In the mathematical literature, tumor growth has been modelled using various  microscopic and macroscopic models \cite{BD}. At the macroscopic level, we may distinguish between models which describe the tumor growth through the dynamics of its cell density, and free boundary models in which tumor is described by its geometric domain subjected to mechanical constrains \cite{friedman}.
From the mechanical viewpoint living tissues may be considered as fluids \cite{Chaplain}. In the simplest approach the dynamics of cell density is governed by cell division and mechanical pressure. Depending on the modelling assumption, and the complexity of the model, mechanical pressure is incorporated in the fluid velocity through Darcy's law, Stokes' law, Brinkman's law or Navier-Stokes' law (see e.g. \cite{Sherratt, Roose, lowengrub,bordeaux,CaiWu}). Notice that Darcy's law, Stokes's law or Brinkman's law may be derived at least formally from Navier-Stokes' law, and so, the latter may be considered as a generalization of the other models.

In this paper we perform mathematical analysis of the Navier-Stokes model with the growth term as for the models of tumor. We are particularly interested in the stiff pressure law limit, often referred to as incompressible limit. The limiting model is a free boundary compressible/incompressible system of fluid equations. 
Derivation of the free boundary models from cell mechanical models has been the subject of many recent contributions in the field of tumor growth modelling \cite{PQV,PQTV,PV,KPS,KP,MPQ}. These models identify tumor with an area of the incompressible (constrained) fluid, while the surrounding healthy tissue can be viewed as a compressible (unconstrained) fluid. 
In almost all aforementioned  works, for simplicity, Darcy's law is used as a closure relation for the system. This means that the velocity is proportional to the gradient of the mechanical pressure, which results in a porous-like type of system. In \cite{PV}, Birkman's law was used to model the tumor as a visco-elastic medium, see also \cite{BiWa}, and \cite{Prost} for the model of growth of tissue in which cell division and apoptosis introduce stress sources that, in general, are anisotropic.
The aim of this work is to extend these works to more general relation between the velocity and the pressure, namely the Navier-Stokes equation. The unknowns are $\rho(t,x)$, the cell density, and  $\vu(t,x)$, the macroscopic velocity field, depending on the time $t>0$ and the position $x\in \RR^d$. 
Our starting system reads as follows
\begin{subequations}\label{sys:main}
\begin{align}
& \pa_t \rho + \dv(\rho \vu) = \rho G(p),  \label{eq:NS1} \\
& \pa_t(\rho\vu) + \dv\lr{\vr\vu\otimes\vu} - \mu \Delta \vu - \xi \nabla \dv \vu + \nabla p = \vr\vu G(p),  \label{eq:NS2}
\end{align}
 \end{subequations}
where $p$ is the pressure and $\mu>0$, $\mu+\xi>0$ are the viscosity coefficients.
For future use, we provide also the nonconservative form of the equation \eqref{eq:NS2}, which reads
\begin{equation}\label{eq:NS2nc}
\vr \left(\pa_t \vu + \vu \cdot \nabla \vu\right) - \mu \Delta \vu - \xi \nabla \dv \vu + \nabla p = 0.
\end{equation}
The right hand side in \eqref{sys:main} represents the growth term depending
on the pressure,  we assume that
\begin{equation}\label{hypG}
G(p)=G_0(P_M-p),\quad G_0,  P_M>0,
\end{equation}
the quantity $P_M$ is often refered to as the homeostatic pressure \cite{Prost}.
As in \cite{PQV,PQTV,PV}, we choose  the barotropic pressure law
\eq{ p = \rho^\gamma,  \label{eq:p}}
with the exponent $\gamma$ that might be very big. 

The system \eqref{sys:main} is complemented with initial data $\rho(0,x) = \rho^0(x)$, $(\vr\vu)(0,x)=\vc{m}^0(x)$,
which are chosen such that for any large enough $\gamma$,
\eq{\label{hypini}
0 \leq \rho^0 \leq 1,\quad \rho^0\in L^1(\RR^d), \quad p^0=(\rho^0)^\gamma \in L^1(\RR^d),\\
\intRdB{\frac{1}{2}\frac{ |\vc{m}^0|^2}{\vr^0} + \frac{1}{\gamma-1} (\rho^0)^\gamma} <+\infty,
}
uniformly with respect to $\gamma$. 
Moreover, we prescribe the values of $\vu$ and $\rho$ at infinity:
\eq{\label{hypbound} 
\vu\to \vc{0},\  \vr\to 0, \quad \mbox{for}\ |x|\to \infty,}
with the relevant compatibility condition for the initial data.

When $\gamma$ is fixed, the system \eqref{sys:main} is the compressible Navier-Stokes system with additional terms on the right hand side of the continuity equation \eqref{eq:NS1} and in the momentum equation \eqref{eq:NS2}. The purpose of this paper is to rigorously justify the  so-called stiff pressure law limit, i.e. $\gamma\to+\infty$ which leads to the two phase compressible/incompressible system
\begin{subequations}\label{sys:main2}
\begin{align}
& \pa_t \rho_\infty + \dv(\rho_\infty \vu_\infty) = \rho_\infty G(p_\infty), \label{eqlim1}  \\
& \pa_t(\rho_\infty\vu_\infty) +\dv( \rho_\infty\vu_\infty\otimes\vu_\infty) - \mu \Delta \vu_\infty - \xi \nabla \dv \vu_\infty + \nabla p_\infty =  \rho_\infty \vu_\infty G(p_\infty), \label{eqlim2} \\
&0\leq \rho_\infty\leq 1,\label{eqlim2a}\\
& p_\infty (1- \rho_\infty) = 0.  \label{eqlim4}
\end{align}
 \end{subequations}
This system is complemented with the same initial data $\rho^0, \vc{m}^0$ as system \eqref{sys:main}.

 The limit of this type was first considered for the compressible Navier-Stokes equations without any additional growth terms by {\sc Lions} and {\sc Masmoudi} \cite{LM}. 
Similar limit passage was also recently investigated for polymeric fluids \cite{DT17}. We would also like to remark that the two-phase models of the type \eqref{sys:main2} can be obtained as the limit of the compressible Navier-Stokes system with the singular pressure, see \cite{PZ, DMZ}.
Compared to the case without growth term, the main difficulty lies in obtaining strong convergence for the density. Indeed classical approach developed by {\sc Lions} \cite{Lions2} and {\sc Feireisl} \cite{EF2001} fails precisely due to the presence of the growth term. Therefore, we follow a recent strategy proposed by {\sc Bresch \& Jabin} \cite{BJ} for the compressible Navier-Stokes equation (see also \cite{BJ_short}) and adapt it to the case at hand.
 
Before stating our main result, let us explain formally  how the system \eqref{sys:main2} may be obtained from  \eqref{sys:main}. We assume existence of a sequence denoted by $n$, such that for $n\to \infty$, $\gamma_n\to \infty$, and $\rho_n\to\rho_\infty$, $\vu_n\to\vu_\infty$, $\vr_n^{\gamma_n}\to p_\infty$ strongly. Writing \eqref{eq:p} as $p_n=\vr_n p_n^{\frac{\gamma_n-1}{\gamma_n}}$ and letting $n\to\infty$ we check that $\rho_\infty$, $p_\infty$ satisfy the relation \eqref{eqlim4}.

 Let us now introduce the set  $\Omega = \{p_\infty >0\} \subset \RR^d$, we have two cases:
\begin{itemize}
\item On $\RR^d\setminus\Omega$, we have $p_\infty=0$, thus \eqref{eqlim1}--\eqref{eqlim2} reduces to
\eq{\label{sys:compressible}
& \pa_t \rho_\infty + \dv(\rho_\infty \vu_\infty) = \rho_\infty G_0P_M,  \\
& \pa_t(\rho_\infty\vu_\infty) + \dv(\rho_\infty\vu_\infty\otimes \vu_\infty) - \mu \Delta \vu_\infty - \xi \nabla \dv \vu_\infty = \rho_\infty \vu_\infty G_0P_M,
}
which is the compressible pressureless Navier-Stokes system with the source term.
\item On $\Omega$, we deduce from \eqref{eqlim4} that we have $\rho_\infty=1$.
  Then \eqref{eqlim1}--\eqref{eqlim2a} reduces to
\eq{\label{sys:incompressible}
& \dv \vu_\infty = G(p_\infty),\\
& \pa_t\vu_\infty + \vu_\infty\cdot \nabla \vu_\infty - \mu \Delta \vu_\infty - \xi \nabla \dv \vu_\infty + \nabla p_\infty = \vc{0},  }
which might be seen as the incompressible Navier-Stokes system.
Note that from the expression of $G$ in \eqref{hypG}, we may rewrite the last system as
\begin{align*}
& \pa_t\vu_\infty + \vu_\infty\cdot \nabla \vu_\infty - \mu \Delta \vu_\infty - (\xi+\frac{1}{G_0}) \nabla \dv \vu_\infty = \vc{0},   \\
& p_\infty = P_M - \frac{1}{G_0} \dv \vu_\infty.
\end{align*}
\end{itemize}

Therefore the limit system \eqref{sys:main2} reveals the features of both: compressible and incompressible fluid equations with the free interphase separating $\Omega$ from $\RR^d\setminus\Omega$.

We conclude the introduction by explaining the link between the system \eqref{sys:main2} and the Hele-Shaw system for tumor growth. Neglecting the acceleration term and assuming that the viscous resisting force is proportional to the velocity, then the momentum equation in system \eqref{sys:incompressible} reduces to
$$
\nu_0 \vu_\infty + \nabla p_\infty = 0.
$$
This is the so-called Darcy's law. Inserting this equation into the first equation in \eqref{sys:incompressible}, we recover the Hele-Shaw system for tumor growth, 
$
- \Delta p_\infty = \nu_0 G(p_\infty) \mbox{ on } \Omega.
$

\section{The main result}
Our main result concerns the convergence of weak solutions of system \eqref{sys:main} to weak solutions of the system  \eqref{sys:main2}. Before formulating the main theorem let us first specify the notion of solutions.
\begin{definition}
[{Weak solution of the primitive system}]\label{weaksol}
Suppose that the initial conditions are as in (\ref{hypini}).
{Let $T>0$}.
We say that the couple  $(\vr,\vu)$ is a weak solution of the problem \eqref{sys:main}, \eqref{hypG}, \eqref{eq:p}, {on $[0,T]$} with the boundary conditions \eqref{hypbound}, if 
\begin{equation*}
(\vr,\vu) \in  L^\infty(0,T; L^\gamma({ \RR^d})) \times L^2(0,T;W^{1,2}(\RR^d)),
\end{equation*}
and we have:
\begin{description}
\item{(i)}
$ \vr \in C_w([0,T];L^\gamma(\RR^d))$, and \eqref{eq:NS1} is satisfied in the weak sense
\eq{
\intO{ \vr(T,\cdot) \varphi(T,\cdot)}- \intO{\vr^{0} \varphi(0,\cdot)}=
 \intTO{ \Big(\vr \partial_t \varphi + \vr \vu \cdot \nabla \varphi+\rho G(p)\varphi\Big)},
}
for all test functions $ \varphi \in C_c ^1([0,T]\times\RR^d)$;
\item{(ii)} $ \vr \vu \in C_w([0,T];L^{\frac{2\gamma}{\gamma+1}}(\RR^d))$, and (\ref{eq:NS2}) is satisfied in the weak sense
\eq{
&\intO{ (\vr \vu) (T,\cdot) \cdot \bm{\psi}(T,\cdot)}- \intO{ \vc{m}^0 \cdot \bm{\psi}(0,\cdot)}\\
&\quad=
\intTOB{\vr \vu \cdot \partial_t \bm{\psi}  + \vr \vu \otimes \vu : \nabla \bm{ \psi} }+\intTO{\vr G(p)\vu\cdot\bm{\psi}} \\
&\qquad+ \intTO{p(\vr)\Div \bm{\psi}} - \intTOB{\mu\Grad\vu : \Grad\bm{ \psi} +\xi \,\Div\vu \,\Div\bm{\psi}}  ,
}
for all test functions $\bm{\psi} \in C_c^1([0,T] \times \RR^d)$;
\item{(iii)}
{there exists a constant $C>0$} such that the energy inequality
\eq{
{ {\cal E}(t)  + \int_{0}^t{ \calJ(s)}{\rm d}s\leq \lr{{\cal E}(0)+Ct}e^{G(0)t}}
}
holds for a.a ${t\in (0,T)}$, where
\begin{align}
&{\cal E}(t)={\cal E}(\vr,\vu)(t) = \intO{\Big(\frac 12 \vr |\vu|^2 +\frac{1}{\gamma-1}\rho^\gamma\Big)},\label{energy} \\
&\calJ(t)=\calJ(\vr,\vu)(t) = \intRdB{\mu |\nabla \vu|^2 + \xi (\dv \vu)^2}. \label{dissip}
\end{align}
\end{description}
\end{definition}
\begin{definition}
[Weak solution to the limiting system \eqref{sys:main2}]\label{weaksollim}
Let $T>0$. Let $\rho^0 \in L^1(\RR^d)$ such that $0\leq \rho^0 \leq 1$, $\int_{\RR^d} \frac{|m^0|^2}{\rho^0}\,dx <+\infty$. 
We say that the triple $(\vr_\infty,\vu_\infty,p_\infty)$ is a weak solution to \eqref{sys:main2} on $[0,T]$ with initial condition $\rho(0,\cdot) = \rho^0$, $(\vr\vu)(0,\cdot)=\vc{m}^0$, if
$$
\vr_\infty \in  L^1\cap L^\infty((0,T)\times\RR^d),\quad \vu_\infty \in L^2(0,T;W^{1,2}(\RR^d)), 
\quad p_\infty \in L^2((0,T)\times\RR^d),
$$
and we have:
\begin{description}
\item{(i)} equation \eqref{eqlim1} is satisfied in the weak sense
\eq{
\intO{ \vr_\infty(T,\cdot) \varphi(T,\cdot)}- \intO{\vr^{0} \varphi(0,\cdot)}=
 \intTO{ \Big(\vr_\infty \partial_t \varphi + \vr_\infty \vu_\infty \cdot \nabla \varphi+\rho_\infty G(p_\infty)\varphi\Big)},
}
for all test functions $ \varphi \in C_c^1([0,T]\times\RR^d)$;
\item{(ii)} equation \eqref{eqlim2} is satisfied in the weak sense
\eq{
&\intO{ (\vr_\infty \vu_\infty) (T,\cdot) \cdot \bm{\psi}(T,\cdot)}- \intO{ \vc{m}^0 \cdot \bm{\psi}(0,\cdot)}\\
&\quad=
\intTOB{\vr_\infty \vu_\infty \cdot \partial_t \bm{\psi}  + \vr \vu_\infty \otimes \vu_\infty : \nabla \bm{ \psi} }+\intTO{\vr_\infty G(p_\infty)\vu_\infty\cdot\bm{\psi}} \\
&\qquad+ \intTO{p_\infty\Div \bm{\psi}} - \intTOB{\mu\Grad\vu_\infty : \Grad\bm{ \psi} +\xi \,\Div\vu_\infty \,\Div\bm{\psi}},
}
for all test functions $\bm{\psi} \in C_c^1([0,T] \times \RR^d)$;
\item{(iii)} equations \eqref{eqlim2a} and \eqref{eqlim4} hold a.e. in $(0,T)\times\RR^d$.
\end{description}
\end{definition}

The compactness of the sequence of weak solutions to system \eqref{sys:main} with $\gamma_n\to \infty$ is guaranteed by our main theorem.
\begin{theorem}\label{th:main}
{Let $T>0$.}
Let $\gamma_n$ be such that $\gamma_n\to \infty$ for $n\to\infty$. Let $\{(\vr_n,\vu_n)\}_{n=1}^\infty$ be a sequence of weak solutions to system \eqref{sys:main} with $p(\vr_n)=\vr_n^{\gamma_n}$, in the sense of Definition \ref{weaksol}.
Then, up to extraction of a subsequence, the limit of  $\{(\vr_n,\vu_n, \vr_n^{\gamma_n})\}_{n=1}^\infty$ for $n\to\infty$ solves \eqref{sys:main2} in the sense of {Definition \ref{weaksollim}}.
More precisely, there exist $\rho_\infty,\vu_\infty,p_\infty$ such that:
\eqh{
&0\leq \vr_\infty\leq 1,\\
&\vr_n\to\vr_\infty\quad\mbox{strongly in } L^q((0,T)\times\RR^d),\ \mbox{for any}\  q\geq 1,\\
&\vu_n\rightharpoonup\vu_\infty\quad\mbox{weakly in } L^2(0,T; H^1_{loc}(\RR^d)),\\
&\vr^{\gamma_n}_n\rightharpoonup p_\infty\quad\mbox{weakly in } L^2((0,T)\times\RR^d) .
}
\end{theorem}
The existence of solutions to the primitive system \eqref{sys:main} is a combination of nowadays classical techniques and compactness argument used in the proof of Theorem \ref{th:main}, therefore it is postponed and only roughly discussed in the end of the paper in Section \ref{sec:existence}. Otherwise, the paper is organized as follows. In Section \ref{sec:estimates} we derive the a-priori estimates, i.e. the estimates that can be obtained for the weak solutions of system \eqref{sys:main} and are uniform with respect to parameter $\gamma$. Then, in Section \ref{sec:compactness} we present the main compactness argument implying the pointwise convergence of the sequence $\vr_n$. In Section \ref{sec:limit}, we show that the a-priori estimates and the compactness argument are sufficient to pass to the limit in \eqref{sys:main} to obtain \eqref{sys:main2} which finishes the proof of Theorem \ref{th:main}.

\section{A-priori estimates}\label{sec:estimates}
The estimates presented in this section are derived using the assumption that $(\vr_n,\vu_n)$ is sufficiently smooth solution of \eqref{sys:main}. This is not necessarily true for the weak solutions from Definition \eqref{weaksol}. However, the calculations can be made rigorous on certain level of approximation discussed in Section \ref{sec:existence}.
\subsection{The energy estimate}
Let us denote the energy and the energy dissipation \eqref{energy}, \eqref{dissip} corresponding to $\rho_n, \vu_n$, and $p(\rho_n)=\rho_n^\gamman$ by $\calE_n$, $\calJ_n$, respectively. The following a-priori estimates are then uniform with respect to $n$.

\begin{lemma}\label{lem:nrj}
Under assumptions \eqref{hypini} and \eqref{hypG}, let $T>0$ be fixed,
then we have the following a priori estimates, uniform in $n\in\NN$:
\begin{itemize}
\item[(i)] For all $t\in [0,T]$, $0\leq \rho_n(t)$ and
$\intRd{ \rho_n(t,x)} \leq e^{G_0P_M t} \intRd{ \rho^0(x)}.$
\item[(ii)] There exists a nonnegative constant $C$ (uniform in $n$) such that
for all $t\in [0,T]$
\eq{\label{est:energy}
\calE_n(t) + \int_0^t \calJ_n(s)\,{\rm{d}}s \leq (\calE_n(0) + Ct) e^{G_0P_M t},
}
with $\calE_n(t) $ and $\calJ_n(t) $ defined in \eqref{energy}, \eqref{dissip}.
\item[(iii)] For all $q\in (1,\gamma_n)$, the sequence $(\rho_n)_{n\in \NN}$ is uniformly bounded in $L^\infty(0,T;L^q(\RR^d))$.
\end{itemize}
\end{lemma}
\begin{proof}
The proof of $(i)$ is standard. By Stampacchia method we show that the nonnegativity principle holds; since $\rho^0(x) \geq 0$, we deduce that $\rho_n(t,x) \geq 0$ for any time $t>0$. Thus $p_n(t,x)=\rho_n^{\gamma_n} (t,x)\geq 0$. Then by a simple integration of \eqref{eq:NS1}
$$
\frac{d}{dt} \intRd{\rho_n(t,x)} = \intRd{ \rho_n(t,x) G(p_n(t,x))}
\leq G_0P_M \intRd{ \rho_n(t,x)},
$$
where we use \eqref{hypG}. We then conclude by integration in time and the Gronwall inequality.

For part $(ii)$, we compute
\begin{align*}
\frac{d}{dt} \calE_n(t) & = \intRd{ \pa_t \rho_n\left(\frac 12 |\vu_n|^2 + \frac{\gamma_n}{\gamma_n-1} \rho_n^{\gamma_n-1}\right)} + \intRd{ \rho_n \pa_t \vu_n\cdot \vu_n}  \\
& = \intRdB{\frac 12 \rho_n |\vu_n|^2 G(p_n) + \frac{\gamma_n}{\gamma_n-1} \rho_n^{\gamma_n} G(p_n)} +
\intRd{ \rho_n \vu_n \cdot \Big(\frac 12 \nabla |\vu_n|^2 + \gamma_n \rho_n^{\gamma_n-2} \nabla \rho_n\Big)}  \\
& \quad + \intRd{(\mu \Delta \vu_n + \xi \nabla \dv\vu_n - \nabla p_n - \rho_n \vu_n \cdot \nabla \vu_n)\cdot \vu_n },
\end{align*}
where we used \eqref{eq:NS1} for the first two terms and {\eqref{eq:NS2nc}}
for the last term. Noticing that 
$\nabla p_n = \gamma_n {\rho_n^{\gamma_n-1}} \nabla \rho_n$, and 
$\vu_n\cdot \nabla |\vu_n|^2 = 2 (\vu_n\cdot\nabla \vu_n)\cdot \vu_n$,
we may cancel the second 
integral with the last two terms of the last integral. Then, 
integrating by parts, we deduce
\begin{equation}\label{estim1E}
\frac{d}{dt} \calE_n(t) + \calJ_n(t) = \int_{\RR^d} \left(\frac 12 \rho_n |\vu_n|^2 G(p_n) + \frac{\gamma_n}{\gamma_n-1} \rho_n^{\gamma_n} G(p_n)\right)\,dx.
\end{equation}
Since $p_n\geq 0$, and $G$ satisfies \eqref{hypG},
we have $G(p_n)\leq G_0P_M$, then
$$
\frac{d}{dt} \calE_n(t) + \calJ_n(t) \leq G_0P_M \calE_n(t) +
\intO{\frac{\gamma_n}{\gamma_n-1} \rho_n^{\gamma_n} G(p_n)}.
$$
Moreover, still using assumption \eqref{hypG}, we have that $G(p_n)\leq 0$
if $p_n\geq P_M \iff  \rho_n\geq P_M^{1/\gamma_n}$. Then
\begin{align*}
\intO{\rho_n^{\gamma_n} G(p_n)} & \leq
\intO{\mathbf{1}_{\{\rho_n \leq P_M^{1/\gamma_n}\}} \rho_n^{\gamma_n} G(p_n)}
\leq G_0P_M \intO{ \mathbf{1}_{\{\rho_n \leq P_M^{1/\gamma_n}\}} P_M^{1-1/\gamma_n} \rho_n } \\
& \leq G_0 P_M^{2-1/\gamma_n} \|\rho_n(t)\|_{L^1(\RR^d)}.
\end{align*}
Thus, using the bound on the $L^1$ norm of $\rho_n$ from part $(i)$, we deduce that there exists a nonnegative constant $C$ such that uniformly in $n$ we have
$$
\frac{d}{dt} \calE_n(t) + \calJ_n(t) \leq G_0P_M \calE_n(t) + C e^{G_0P_Mt},
$$
and we conclude using the Gronwall lemma.

$(iii)$ As a consequence of the point (ii) above, we deduce 
$$
\intO{\rho_n^{\gamma_n}(t,x)} \leq \gamma_n (\calE_n(0) + C t) e^{G_0P_Mt}.
$$
Then,
$$
\sup_{t\in(0,T)}\|\rho_n(t)\|_{L^{\gamma_n}(\RR^d)} \leq \left(\gamma_n (\calE_n(0) + C T) e^{G_0P_MT}\right)^{1/\gamma_n} \to 1, \qquad\mbox{ as } n\to + \infty.
$$
By interpolation, for any $q\in (1,\gamma_n)$, we have
\begin{equation}\label{estimrhoq}
\|\rho_n(t)\|_{L^q(\RR^d)} \leq \|\rho_n(t)\|_{L^1(\RR^d)}^{\theta_n} \|\rho_n(t)\|_{L^{\gamma_n}(\RR^d)}^{1-\theta_n}
\leq \|\rho^0\|_{L^1(\RR^d)}^{\theta_n} e^{\theta_n G_0P_Mt}
\left(\gamma_n (\calE_n(0) + C t) e^{G_0P_Mt}\right)^{\frac{1-\theta_n}{\gamma_n}},
\end{equation}
with $\frac 1q = \theta_n + \frac{1-\theta_n}{\gamma_n}$, then when $n\to+\infty$, we have $\theta_n\to \frac 1q$. We deduce that for $N$ large enough, the sequence $\{\rho_n\}_{n\geq N}$ is uniformly bounded in $L^\infty(0,T;L^q(\RR^d))$.
\end{proof}

Lemma \ref{lem:nrj} implies that we may extract a subsequence (still labelled by $n$), such that ${\rho_n}\rightharpoonup \rho_\infty$ weakly as $n\to +\infty$.
Then, by lower semi-continuity of norm, passing into the limit $n\to +\infty$ in \eqref{estimrhoq}, we have
$$
\|\rho_\infty\|_{L^\infty(0,T;L^q(\RR^d))} \leq \underset{n\to+\infty}{\lim\inf} \|\rho_n\|_{L^\infty(0,T;L^q(\RR^d))} \leq \|\rho^0\|_{L^1(\RR^d)}^{1/q} e^{G_0P_MT/q}.
$$
Since this latter estimate is true for any $q\geq 1$, we may let $q$ going to $+\infty$ to find
$$
\|\rho_\infty\|_{L^\infty((0,T)\times\RR^d)} \leq \underset{q\to+\infty}{\lim\inf} \|\rho_\infty\|_{L^\infty(0,T;L^q(\RR^d))} \leq 1.
$$
In addition, we can prove
\begin{lemma}\label{lem:rhoinf1}
Under the same assumptions as in Lemma \ref{lem:nrj}, we have $(\rho_n-1)_+ \to 0$ as $n\to +\infty$.
\end{lemma}
\begin{proof}
Let us introduce $\phi_n = (\rho_n-1)_+$. From the energy estimate \eqref{est:energy}, we deduce 
$$
\intO{(1+\phi_n)^{\gamma_n} \mathbf{1}_{\{\phi_n>0\}} }
\leq \intO{\rho_n^{\gamma_n}} \leq
(\calE(0)+Ct) e^{G_0P_Mt} \gamman.
$$
It has been proved in \cite[p.24]{LM} that, for any $q>1$ and any $x\geq 0$, there exists $a_q>0$
such that $(1+x)^k \geq 1+a_q k^q x^q$, for $k$ large enough. 
Thus, for sufficiently large $n$, we have
$$
\intO{\phi_n^q } \leq \frac{(\calE_n(0)+Ct) e^{G_0P_Mt}}{a_q \gamma_n^{q-1}} \underset{n\to +\infty}{\longrightarrow} 0.
$$
Therefore $\phi_n \to 0$ strongly in $L^q((0,T)\times\RR^d)$ for any $q\geq 1$ .
\end{proof}

\subsection{The estimate of the pressure}
Note that the energy estimate from the previous section does not provide any estimate of the pressure independent of $n$, only the estimate of the density. In the following lemma, we state an $L^2$ estimate on the pressure.
\begin{lemma}\label{lem:pL2}
Under the same assumptions as in Lemma \ref{lem:nrj}, the sequence $\{p_n\}_{n=1}^\infty$ is uniformly bounded in $L^2((0,T)\times\RR^d)$.
\end{lemma}
\begin{proof}
From the renormalization property \cite{DL} for equation \eqref{eq:NS1},
we have that for any $C^1$ function $\beta$: $\RR\to\RR$ such that
$|\beta(y)|\leq C(1+y)$,
\begin{equation}\label{eq:betarn}
\pa_t \beta(\rho_n) + \dv(\beta(\rho_n) \vu_n) = (\beta(\rho_n)-\rho_n \beta'(\rho_n)) \dv \vu_n
+ \rho_n \beta'(\rho_n) G(p_n).
\end{equation}
Let $K>0$ be a nonnegative constant. We define, for $\gamma_n>1$, $\beta_K$ the function
$$
\beta_K(y) = \left\{ \begin{array}{ll}
0, \qquad  & \mbox{ if } y\leq 0,  \\
y^{\gamma_n}, \qquad  & \mbox{ if } y\in (0,K),  \\
\gamma_n K^{\gamma_n-1} y + K^{\gamma_n}(1-\gamma_n), \qquad  & \mbox{ if } y \geq K.  \\
\end{array}\right.
$$
For all $y\geq 0$ and $\gamma_n>1$, we have
\begin{equation}\label{estbetaK}
0 \leq \beta_K(y) \leq y \beta_K'(y) \quad \text{and}\quad
\gamma_n y\beta_K'(y) y^{\gamma_n} \geq \big( y \beta_K'(y)\big)^2.
\end{equation}
Using \eqref{eq:betarn} with $\beta=\beta_K$ and inserting the assumption \eqref{hypG} on the growth function $G$, we deduce
\begin{align*}
\pa_t \beta_K(\rho_n) + \dv(\beta_K(\rho_n) \vu_n) + G_0 \rho_n \beta_K'(\rho_n) p_n  & = 
\rho_n \beta_K'(\rho_n) G_0 P_M + (\beta_K(\rho_n)-\rho_n\beta_K'(\rho_n)) \dv \vu_n  \\
& \leq \rho_n \beta_K'(\rho_n) G_0 P_M + 2\rho_n\beta_K'(\rho_n) |\dv \vu_n|,
\end{align*}
where we used \eqref{estbetaK} to get the last inequality.
On the set $\{\rho_n\leq 1\}$, we clearly have $\rho_n \beta_K'(\rho_n)\leq \gamma_n \rho_n$. Then,
\begin{align*}
\pa_t \beta_K(\rho_n) + \dv(\beta_K(\rho_n) \vu_n) + G_0 \rho_n \beta_K'(\rho_n) p_n  \leq &
\ \gamma_n \rho_n G_0 P_M \mathbf{1}_{\{\rho_n\leq 1\}}  \\
& \ + \rho_n\beta_K'(\rho_n) \big(2|\dv \vu_n|+G_0 P_M \rho_n \mathbf{1}_{\{\rho_n > 1\}}\big).
\end{align*}
Integrating and using the Cauchy-Schwarz and the Young inequalities, we deduce that for any $\epsilon>0$, there exists a nonnegative constant $C_\epsilon$ such that
\eqh{
&\frac{d}{dt}\intO{\beta_K(\rho_n)} + G_0 \intO{ \rho_n \beta_K'(\rho_n) p_n} \\
&\quad\leq \gamma_n G_0P_M \intO{ \rho_n} 
+ \gamma_n \epsilon \intO{ \frac{\big(\rho_n\beta_K'(\rho_n)\big)^2}{{\gamma_n}^2} } 
+ \gamma_n C_\epsilon \intO{(\dv \vu_n^2 + \rho_n^2)}.
}
We may fix $\epsilon>0$ such that, from \eqref{estbetaK},
$$
\gamma_n \epsilon \intO{ \frac{\big(\rho_n\beta_K'(\rho_n)\big)^2}{{\gamma_n}^2} } \leq 
\frac{G_0}{2} \intO{ \rho_n \beta_K'(\rho_n) p_n}.
$$
Integrating in time, we obtain that there exists a nonnegative constant $C$ such that
\begin{align*}
&\|\beta(\rho_n)(T)\|_{L^1(\RR^d)} +
\frac{G_0}{2} \intTO{ \rho_n \beta_K'(\rho_n) p_n}  \\
&\leq  \|\beta(\rho_n)(0)\|_{L^1(\RR^d)} + 
C \gamma_n \left(\|\rho_n\|_{L^1((0,T)\times\RR^d)} + \|\rho_n\|^2_{L^2((0,T)\times\RR^d)}
+ \|\dv \vu_n\|^2_{L^2((0,T)\times\RR^d)}\right).
\end{align*}
Using estimates in Lemma \ref{lem:nrj}, we deduce that there exists an uniform (with respect to $n$ and $K$) constant $C>0$ such that
$$
\frac{1}{\gamma_n}\intTO{ \rho_n\beta_K'(\rho_n) p_n}
\leq C.
$$
Therefore, for all $K\geq 0$, we deduce that
$$
\intTO{ p_n^2 \mathbf{1}_{\{\rho_n\leq K\}} } \leq C.
$$
We may now let $K$ go to $+\infty$ and, by the monotone convergence theorem, we conclude the proof. 
\end{proof}

\begin{remark}
As a consequence of Lemma \ref{lem:pL2}, we deduce that $\rho_n$ is bounded in 
$L^{2\gamma_n}([0,T]\times\RR^d)$. Then we may apply Lemma 6.9 from \cite{NS} 
and deduce that \eqref{eq:betarn} holds with  $\beta(y)=y^\gamma$. We therefore obtain the evolution equation for the  pressure $p_n=\vr_n^\gamman$,
\begin{equation}\label{edp:p}
\pa_t p_n + \vu_n\cdot\nabla p_n + \gamma_n p_n \dv \vu_n = \gamma_n p_n G(p_n).
\end{equation}
\end{remark}

\begin{remark}
The fact that we can derive the uniform estimates for the pressure is one of the main advantages of the growth term in the continuity equation \eqref{eq:NS1}. Not having it would require more laborious estimates with the application of Bogovski type of operator, see for example \cite{LM}, \cite{PZ}.
\end{remark}

\subsection{The estimate of the nonlinear terms in the momentum equation}
Before letting $n\to \infty$, we need to provide the uniform estimates of the rest of nonlinear terms from the momentum equation \eqref{eq:NS2}. Applying the operator $(-\Delta)^{-1} \dv$ to both sides of {\eqref{eq:NS2nc}}, we deduce
\begin{equation}\label{link:up}
(\mu + \xi) \dv \vu_n = p_n -(- \Delta)^{-1}\big(\dv(\rho_n\pa_t \vu_n + \rho_n \vu_n\cdot\nabla \vu_n))\big) = p_n + \calD(\rho_n \vu_n),
\end{equation}
where we use the notation for the total derivative 
$$\calD(\rho_n \vu_n) =-(- \Delta)^{-1}\big(\dv(\rho_n \pa_t \vu_n + \rho_n \vu_n\cdot \nabla \vu_n))\big).$$
Using the $L^2$ estimate of the pressure and  the $L^2$ estimate of $\dv \vu_n$ following from the energy estimate we deduce the following fact.
\begin{corollary}\label{cor:D}
Under the assumptions of Lemma \ref{lem:pL2}, the sequence $\{\calD(\rho_n \vu_n)\}_{n=1}^\infty$ is uniformly bounded in $L^2([0,T]\times\RR^d)$.
\end{corollary}

\section{Compactness}\label{sec:compactness}
The purpose of this section is to establish the compactness of the density sequence $\{\rho_n\}_{n=1}^{\infty}$. To do it, 
we follow the strategy proposed by {\sc Bresch \& Jabin} \cite{BJ}  (see also \cite{BJ_short}) in the context of compressible Navier-Stokes equations with the non-monotone pressure law. We adapt their approach to  whole space $\RR^d$ case, with 
a nonzero growth term in the right hand side of the continuity equation, and consequently, the conservative form of the momentum equation. Application of nowadays classical approach developed by {\sc Lions} \cite{Lions2} and {\sc Feireisl} \cite{EF2001} fails precisely due to the presence of this additional term.

The main result of this section is the following 
\begin{proposition}\label{prop:rho}
Let $T>0$.
Assume that $\{(\rho_n,\vu_n)\}_{n=1}^{\infty}$ satisfies \eqref{sys:main}, \eqref{eq:p} with assumptions  \eqref{hypG}, \eqref{hypini}, such that the estimates from Lemma \ref{lem:nrj} and in Lemma \ref{lem:pL2} hold.

Then the sequence $\{\rho_n\}_{n=1}^\infty$ is compact in $L^2_{loc}([0,T]\times\RR^d)$.
\end{proposition}

The rest of this section is dedicated to the proof of this fact.

\subsection{A compactness criterion}
In order to prove local compactness for the density sequence $\{\rho_n\}_{n=1}^\infty$
we use a compactness criterion, for the proof of which we refer the reader to \cite[Lemma 3.1]{Belgacem}, or \cite[Proposition 4.1]{BJ}. This criterion was applied to the study of Navier-Stokes equations with non-monotone pressure and anisotropic stress tensor in the aforementioned papers \cite{BJ,BJ_short}.

Let us first introduce the necessary notations.

We define a family $\{K_h\}_{h>0}$ of nonnegative function by 
$$\ds K_h(x)=\frac{1}{(|x|^2+h^2)^{d/2}}$$ 
for $|x|\leq 1$. Otherwise, $K_h$ belongs to $C^\infty(\RR^d\setminus B(0,1))$ and is compactly supported
in $B(0,2)$. Moreover $K_h$ is equal to some function $K(x)$ independent on $h$ outside $B(0,3/2)$.
We will also make use of the inequality 
\begin{equation}\label{ineqK}
|x||\nabla K_h(x)| \leq C K_h(x),
\end{equation}
which holds for some nonnegative constant $C$ independent of $h$,
thanks to our choice for $K_h$.
We also denote
$$
\overline{K_h}(x) = \frac{K_h(x)}{\|K_h\|_{L^1(\RR^d)}},
\qquad
\calK_{h_0}(x) = \int_{h_0}^1 \overline{K_h}(x) \frac{dh}{h}.
$$

Then  the compactness criterion states what follows.
\begin{lemma}\label{lem:compact}
Assume $\{\rho_n\}_{n=1}^\infty$ is a sequence of functions uniformly bounded in $L^q((0,T)\times\RR^d)$
with $1\leq q<+\infty$. 
If $\{\pa_t \rho_n\}_{n=1}^\infty$ is uniformly bounded in $L^r([0,T],W^{-1,r}(\RR^d))$ with 
$r\geq 1$ and  
$$
\underset{n}{\lim\sup} \left( \frac{1}{\|K_h\|_{L^1}} \iintO{
K_h(x-y)|\rho_n(x)-\rho_n(y)|^q}\right) \to 0, \quad \mbox{ as } h\to 0. 
$$
Then, $\{\rho_n\}_{n=1}^\infty$ is compact in $L^q_{loc}([0,T]\times\RR^d)$.
Conversely, if $\{\rho_n\}_{n=1}^\infty$ is compact in $L^q_{loc}([0,T]\times\RR^d)$, 
then the above $\lim\sup$ converges to $0$ as $h$ goes to $0$.
\end{lemma}

\subsection{Definition of the weights}

Let us define the weights $w_n$ as solutions of the transport equation
\begin{equation}\label{eq:weight}
\pa_t w + \vu_n \cdot \nabla w = -\lambda B_n w, \qquad B_n = M|\nabla \vu_n|, 
\end{equation}
complemented with the initial data $w(t=0)=1$. 
Here $\lambda$ is some nonnegative constant which will be fixed later on.
To simplify the notations, we drop the index $n$ denoting the weight simply by $w$.

By $M$ we denote the maximal operator, defined by
$$
M f(x) = \sup_{r\geq 1} \frac{1}{|B(0,r)|} \int_{B(0,r)} f(x+z)\,dz.
$$
Recall that we have the following inequality (see e.g. \cite{Stein})
$$
|\Phi(x)-\Phi(y)| \leq C |x-y| (M|\nabla \Phi|(x) + M|\nabla \Phi|(y)),
$$
for any $\Phi$ in $W^{1,1}(\RR^d)$.
Note that, thanks to Lemma \ref{lem:nrj} and Lemma \ref{lem:pL2},
we have that $B_n$ defined in \eqref{eq:weight} is uniformly bounded in $L^2([0,T]\times\RR^d)$. This allows us to deduce the following properties of the weight $w$.

\begin{proposition}\label{prop:w}
Let us assume that $\vu_n$ is given and that it is  bounded in $L^2_{loc}([0,T]\times\RR^d)\cap L^\infty(0,T;H^1(\RR^d))$ uniformly with respect to $n$.
Then, there exists a unique solution to \eqref{eq:weight}. Moreover, we have
\begin{itemize}
\item[(i)] For any $(t,x)\in(0,T)\times\RR^d$, $0\leq w(t,x) \leq 1$.
\item[(ii)] If we assume moreover that the pair $(\rho_n,\vu_n)$ is a solution to \eqref{eq:NS1} and $\vr_n$
is uniformly bounded in $L^2([0,T]\times\RR^d)$, there exists $C\geq 0$, such that
\begin{equation}\label{boundlogw}
\intO{ \rho_n |\log w| } \leq C \lambda.
\end{equation}
\end{itemize}
\end{proposition}

\begin{proof}
$(i)$ Since 
$B_n\in L^2([0,T]\times\RR^d)$, and $\vu_n\in L^2_{loc}([0,T]\times\RR^d)\cap L^\infty(0,T;H^1(\RR^d))$, by standard theory of renormalized solutions to the transport equations  \cite{DL}, 
we may construct a nonnegative
solution to \eqref{eq:weight}. Moreover, since $B_n$ is nonnegative, we have clearly that $w\leq 1$, since it is true initially.

$(ii)$ From part  $(i)$, we have $|\log w|=-\log w$.
By renormalization of equation \eqref{eq:weight}, we have
$$
\pa_t |\log w| + \vu_n \cdot \nabla |\log w| = \lambda B_n.
$$
Therefore, using also the continuity equation \eqref{eq:NS1}, we get
$$
\pa_t (\rho_n |\log w|) + \dv(\rho_n \vu_n |\log w|) = \rho_n |\log w| G(p_n)
+ \lambda \rho_n B_n.
$$
We integrate it in space and use \eqref{hypG} to deduce
$$
\frac{d}{dt} \intO{ \rho_n |\log w|} \leq G_0P_M \intO{ \rho_n |\log w|}
+ \lambda \intO{ \rho_n B_n }.
$$
Using the Gromwall lemma, we obtain
$$
\intO{\rho_n |\log w|(T,x)} \leq \lambda e^{G_0P_M T}  \intTO{ \rho_n B_n}.
$$
Finally, since $B_n$ and $\rho_n$ are uniformly bounded in $L^2((0,T)\times\RR^d)$,
we conclude using the Cauchy-Schwarz inequality.
\end{proof}

\subsection{Propagation of regularity for the transport equation}\label{sec:compactrho}

We first consider the transport equation \eqref{eq:NS1} with the pressure law
\eqref{eq:p} without the coupling through the velocity field $\vu_n$.
Taking the difference of the equations \eqref{eq:NS1} satisfied by $\rho_n(x)$
and $\rho_n(y)$, we get
\begin{align*}
&\pa_t (\rho_n(x)-\rho_n(y)) + \dv_x (\vu_n(x)\lr{\rho_n(x)-\rho_n(y)}) +
\dv_y (\vu_n(y)\lr{\rho_n(x)-\rho_n(y)})  \\
&=\frac 12 (\dv_x \vu_n(x) + \dv_y \vu_n(y)) \lr{\rho_n(x)-\rho_n(y)} \\
&\quad - \frac 12 (\dv_x \vu_n(x)-\dv_y \vu_n(y))(\rho_n(x)+\rho_n(y))  \\
&\quad + \lr{\rho_n(x) G(p_n(x)) -\rho_n(y) G(p_n(y))} .
\end{align*}
multiplying by $(\rho_n(x)-\rho_n(y))$, 
we deduce
\begin{align*}
&\frac{1}{2}\pa_t (\rho_n(x)-\rho_n(y))^2 + \frac{1}{2}\dv_x (\vu_n(x)\lr{\rho_n(x)-\rho_n(y)}^2) +
\frac{1}{2}\dv_y (\vu_n(y)\lr{\rho_n(x)-\rho_n(y)}^2) \\
&= - \frac 12 (\dv_x \vu_n(x)-\dv_y \vu_n(y))(\rho_n(x)+\rho_n(y)) \lr{\rho_n(x)-\rho_n(y)} \\
&\quad + \lr{\rho_n(x) G(p_n(x)) -\rho_n(y) G(p_n(y))} \lr{\rho_n(x)-\rho_n(y)}.
\end{align*}
This computation can be made rigorous using renormalization
technique \cite{DL}.
We observe that thanks to our pressure law in \eqref{eq:p}, we have that
$\mbox{sign }(\rho_n(x)-\rho_n(y)) = \mbox{sign }(p_n(x)-p_n(y)).$
Then, we can rearrange the last term of the right hand side as
\begin{align*}
&(\rho_n(x) G(p_n(x)) -\rho_n(y) G(p_n(y)) \lr{\rho_n(x)-\rho_n(y)}\\
& =
G_0 P_M\lr{\rho_n(x)-\rho_n(y)}^2  -G_0\lr{\rho_n(x)^{\gamman+1}-\rho_n(y)^{\gamman+1}} \lr{\rho_n(x)-\rho_n(y)}  \\
& \leq G_0P_M\lr{\rho_n(x)-\rho_n(y)}^2,
\end{align*}
where we use the definition of $G$ \eqref{hypG}.
Moreover, since $p_n$ is nonnegative, $G(p_n)\leq G_0P_M$. We arrive at
\eq{\label{eq:drho1}
&\frac{1}{2}\pa_t (\rho_n(x)-\rho_n(y))^2 + \frac{1}{2}\dv_x (\vu_n(x)\lr{\rho_n(x)-\rho_n(y)}^2) +
\frac{1}{2}\dv_y (\vu_n(y)\lr{\rho_n(x)-\rho_n(y)}^2) \\
& \leq - \frac 12 (\dv_x \vu_n(x)-\dv_y \vu_n(y))(\rho_n(x)+\rho_n(y)) \lr{\rho_n(x)-\rho_n(y)} \\
&\quad + G_0P_M\lr{\rho_n(x)-\rho_n(y)}^2. 
}
We then introduce
$$
R(t) = \frac12\iintO{K_h(x-y) \lr{\rho_n(x)-\rho_n(y)}^2 (w(x)+w(y))},
$$
and
$$
\calR_{h_0}(t) = \frac12\iintO{\calK_{h_0}(x-y) \lr{\rho_n(x)-\rho_n(y)}^2(w(x)+w(y))}
= \frac{1}{\|K_h\|_{L^1}} \int_{h_0}^1 R(t)\frac{{\rm d}h}{h},
$$
where the weights $w$ satisfy \eqref{eq:weight}.

Using \eqref{eq:drho1} and the symmetry of $K_h$, we deduce
\begin{equation}\label{eqdtR}
\frac{d}{dt} R(t) \leq A_1 + A_2 + A_3 + G_0P_M R(t),
\end{equation}
where
$$
A_1 = \frac12\iintO{\nabla K_h(x-y) (\vu_n(x)-\vu_n(y)) \lr{\rho_n(x)-\rho_n(y)}^2(w(x)+w(y))},
$$
$$
A_2 =  \iintO{ K_h(x-y)\lr{\rho_n(x)-\rho_n(y)}^2(\pa_t w(y) + \vu_n(y)\cdot\nabla w(y) )},
$$
$$
{A_3 = -2 \iintO{ K_h(x-y)(\dv \vu_n(x) - \dv \vu_n(y)) \lr{\vr_n(x)-\vr_n(y)}\rho_n(x)w(x) }.}
$$
\subsubsection*{Estimate of $A_1$}
The term $A_1$ is the same as in \cite{BJ,BJ_short}. For the sake of completeness we recall how to estimate it below. First, we make use of the following inequality
$$
|\vu_n(x)-\vu_n(y)|\leq C|x-y|\big(D_{|x-y|}\vu_n(x)+D_{|x-y|}\vu_n(y)\big),
$$
where $D_h\vu_n(x)=\frac{1}{h} \int_{|z|\leq h} \frac{|\nabla \vu_n(x+z)|}{|z|^{d-1}}\,dz$. 
Recall that $D_n \vu_n \leq M |\nabla \vu_n|$.
For the proof we refer the reader to \cite[Lemma 3.1]{Jabin2010}.
Then, using inequality \eqref{ineqK} and the symmetry of $K_h$ we get
\begin{align*}
A_1 &\leq C\iintO{ |x-y| \nabla K_h(x-y)\big(D_{|x-y|}\vu_n(x)+
D_{|x-y|} \vu_n(y)\big) \lr{\rho_n(x)-\rho_n(y)}^2(w(x)+w(y))}  \\
&\leq C \iintO{ K_h(x-y)|D_{|x-y|}\vu_n(x)+D_{|x-y|}\vu_n(y)| \lr{\rho_n(x)-\rho_n(y)}^2 w(y)}.
\end{align*}
Next, we integrate in $h$ on $(h_0,1)$. Using that
$$
D_{|x-y|}\vu_n(x)+D_{|x-y|}\vu_n(y) = D_{|x-y|}\vu_n(x)- D_{|x-y|}\vu_n(y) + 2D_{|x-y|}\vu_n(y),
$$ 
and changing the variables $z=x-y$, we may apply the Cauchy-Schwarz inequality
and the uniform $L^4$ bound on $\rho_n$ to deduce
\begin{align*}
\int_{h_0}^1 \frac{A_1}{\|K_h\|_{L^1}}\,\frac{{\rm d}h}{h} \leq & \ C
\int_{h_0}^1  \int_{\RR^{d}} \overline{K_h}(z) \|D_{|z|}\vu_n(\cdot) - D_{|z|}\vu_n(\cdot+z)\|_{L^2} \,{\rm d}z\,\frac{\dh}{h}  \\
& + C \iintO{\calK_{h_0}(x-y) D_{|x-y|} \vu_n(y)\lr{\rho_n(x)-\rho_n(y)}^2 w(y)}.
\end{align*}
We may bound $D_{|x-y|}\vu_n$ by the Maximal operator $M|\nabla \vu_n|$, thus
\eq{\label{estimA1}
\int_{h_0}^1 \frac{A_1}{\|K_h\|_{L^1}}\,\frac{\dh}{h} \leq & \ C
\int_{h_0}^1  \int_{\RR^{d}} \overline{K_h}(z) \|D_{|z|}\vu_n(\cdot) - D_{|z|}\vu_n(\cdot+z)\|_{L^2} \,{\rm d}z\,\frac{\dh}{h}  \\
& + C \iintO{ \calK_{h_0}(x-y) M|\nabla \vu_n(y)|\lr{\rho_n(x)-\rho_n(y)}^2 w(y)}.  
}
The second term on the right hand side of \eqref{estimA1} will be controlled by the term $A_2$.

\subsubsection*{Estimate of $A_2$}
From \eqref{eq:weight}, we have
$$
A_2 = \iintO{K_h(x-y) \lr{\rho_n(x)-\rho_n(y)}^2(-\lambda {B_n(y)} ) w(y) }.
$$
Therefore, combining the latter equality with \eqref{estimA1}, we deduce
\begin{align*}
\int_{h_0}^1 \frac{A_1+A_2}{\|K_h\|_{L^1}}\,\frac{\dh}{h} \leq & \ C
\int_{h_0}^1  \int_{\RR^{d}} \overline{K_h}(z) \|D_{|z|}\vu_n(\cdot) - D_{|z|}\vu_n(\cdot+z)\|_{L^2} \,{\rm d}z\,\frac{\dh}{h}  \\
&+ \iintO{ \calK_{h_0}(x-y) \lr{\rho_n(x)-\rho_n(y)}^2 w(y)
\Big( CM|\nabla \vu_n(y)| - \lambda {B_n(y)} \Big)}.
\end{align*}
From the definition of $B_n$ in \eqref{eq:weight}, we can find $\lambda$
large enough such that 
\begin{align}\label{estimA1B1}
\int_{h_0}^1 \frac{A_1+A_2}{\|K_h\|_{L^1}}\,\frac{\dh}{h} \leq  C \int_{h_0}^1  \int_{\RR^{d}} \overline{K_h}(z) \|D_{|z|}\vu_n(\cdot) - D_{|z|}\vu_n(\cdot+z)\|_{L^2} \,{\rm d}z\,\frac{\dh}{h} 
\end{align}

\subsubsection*{Estimate of $A_3$}
To estimate  the $A_3$ term, we first recall the link between $\dv\vu_n$ and $p_n$ \eqref{link:up}, and the notation 
$\calD(\rho \vu) = -(-\Delta)^{-1}\big(\dv(\rho \pa_t \vu + \rho \vu\cdot \nabla \vu))\big)$.
Then,
\eq{\label{estimD1}
A_3  =&  - 2\iintO{ K_{h}(x-y) (\dv \vu_n(x)-\dv \vu_n(y)) \rho_n(x) \lr{\vr_n(x)-\vr_n(y)} w(x)}\\
=& - \frac{2}{\mu+\xi} \iintO{ K_{h}(x-y)  \lr{p_n(x)-p_n(y)}\lr{\vr_n(x)-\vr_n(y)} \rho_n(x)  w(x)}\\
& - \frac{2}{\mu+\xi} \iintO{ K_{h}(x-y) \Big(\calD(\rho_n \vu_n)(x)-\calD(\rho_n \vu_n)(y)\Big) \lr{\vr_n(x)-\vr_n(y)}  \rho_n(x)  w(x)}.
}
Note that since $p_n=\rho_n^{\gamma_n}$ is increasing with respect to $\rho_n$, we have $(p_n(x)-p_n(y)) \lr{\vr_n(x)-\vr_n(y)}\geq 0$. Therefore, the first term in \eqref{estimD1} has a good sign when moved to the left hand side.

Thus, departing from \eqref{eqdtR} and integrating in $h$, we use 
\eqref{estimA1B1} and \eqref{estimD1} to deduce
\eq{\label{estimdtR0}
\Dt\calR_{h_0}(t) \leq & \ G(0) \calR_{h_0}(t) + C
\int_{h_0}^1 \int_{\RR^{d}} \overline{K_h}(z) \|D_{|z|}\vu_n(\cdot) - D_{|z|}\vu_n(\cdot+z)\|_{L^2(\RR^d)}\, {\rm d}z \,\frac{\dh}{h}  \\
& - \frac{2}{\mu+\xi} \iintO{\calK_{h_0}(x-y) (\calD(\rho_n \vu_n)(x)-\calD(\rho_n \vu_n)(y)) \lr{\vr_n(x)-\vr_n(y)}\rho_n(x)  w(x)}.
}
{The estimate of} the second term in \eqref{estimdtR0} follows from the following Lemma:
\begin{lemma}[Lemma 6.3 in \cite{BJ}]
For any $1<p<+\infty$, there exists $C>0$ such that for any $\vu\in H^1(\RR^d)$,
\begin{equation}\label{ineq:lem63}
\int_{h_0}^1 \int_{\RR^{d}} \overline{K_h}(z) \|D_{|z|}\vu(\cdot) - D_{|z|}\vu(\cdot+z)\|_{L^2(\RR^d)} \,{\rm d}z\, \frac{\dh}{h}
\leq C |\log h_0|^{1/2} \|\vu\|_{H^1(\RR^d)}.
\end{equation}
\end{lemma}
To estimate  the last term in \eqref{estimdtR0}, we use:
\begin{lemma}[Lemma 8.3 in \cite{BJ}]\label{lem83}
Assume that $\pa_t \rho_n + \dv(\rho_n \vu_n)=\vr_n G(p_n)$, and $(\rho_n,\vu_n)$ is such that
$$
\sup_n \Big(\|\rho_n\|_{L^\infty(0,T;L^1(\RR^d)\cap L^\gamma(\RR^d))} + 
\|\rho_n |\vu_n|^2\|_{L^\infty(0,T;L^1(\RR^d))} + \|\nabla \vu_n\|_{L^2((0,T)\times\RR^d)}\Big)
<\infty,
$$
for $\gamma>d/2$, and 
$$
\exists\, \bar{q}>1, \qquad \sup_n \|\pa_t(\rho_n \vu_n)\|_{L^2(0,T;W^{-1,\bar{q}}(\RR^d))} < \infty. 
$$
Consider $\Phi\in L^\infty((0,T)\times \RR^{2d})$ such that 
$$
C_\Phi := \left\|\int_{\RR^d} \overline{K_h}(x-y) \Phi(t,x,y)\,\dy \right\|_{W^{1,1}(0,T;W^{-1,1}_x(\RR^d))}
+ \left\|\int_{\RR^d} \overline{K_h}(x-y) \Phi(t,x,y)\,\dx \right\|_{W^{1,1}(0,T;W^{-1,1}_y(\RR^d))}
$$
is finite.
Then, there exists $\theta>0$ such that
$$
\iintTO{ \overline{K_h}(x-y) (\calD(\rho_n \vu_n)(t,x)-\calD(\rho_n \vu_n)(t,y)) \Phi(t,x,y)} \leq C h^\theta \big(\|\Phi\|_{L^\infty} + C_{\Phi}\big).
$$
\end{lemma}
\begin{remark}\label{rem:Lemma8.3}
The only change in the statement of the above lemma with respect to Lemma 8.3 in \cite{BJ} is that in our case the continuity equation has an extra production term. Note however, that the operator $\calD(\vr\vu)$ is  the Riesz operator applied to the nonconservative form of the momentum transport, see \eqref{link:up}. However, the momentum equation in the nonconservative form {\eqref{eq:NS2nc}} does not include any extra contribution from $G(p)$. This makes the proof of Lemma \ref{lem83} the same as the proof of Lemma 8.3 from \cite{BJ}.
\end{remark}
In order to apply Lemma \ref{lem83}, we need to truncate the integrant in the last integral of \eqref{estimdtR0}. We introduce a smooth truncation function 
$\phi: [0,\infty)\to [0,1]$ such that  $0\leq \phi\leq 1$, $\phi(x)=1$ for $x\leq\frac 12$, and $\phi(x)=0$ for $x>1$.
We then split the last term in \eqref{estimdtR0} into two parts
\eqh{
&-\iintTO{K_h(x-y) (\calD(\rho_n \vu_n)(x)-\calD(\rho_n \vu_n)(y)) \lr{\vr_n(x)-\vr_n(y)}\rho_n(x) w(x)}\\
&=\int_0^T\!\!\!\int_{\RR^{2d}}K_h(x-y) (\calD(\rho_n \vu_n)(x)-\calD(\rho_n \vu_n)(y))\\
&\hspace{2cm}\times \lr{\vr_n(y)-\vr_n(x)}\rho_n(x)  w(x)\lr{1-\phi\Big(\frac{\rho_n(t,x)}{L}\Big)\phi\Big(\frac{\rho_n(t,y)}{L}\Big)}\,\dx\,\dy\,\dt\\
&\quad+\int_0^T\!\!\!\int_{\RR^{2d}}K_h(x-y) (\calD(\rho_n \vu_n)(x)-\calD(\rho_n \vu_n)(y)) \\
&\hspace{2cm}\times\lr{\vr_n(y)-\vr_n(x)} \rho_n(x)  w(x)\phi\Big(\frac{\rho_n(t,x)}{L}\Big)\phi\Big(\frac{\rho_n(t,y)}{L}\Big)\,\dx\,\dy\,\dt.}
Note that for some $\alpha>0$, we have
$$
1-\phi\Big(\frac{\rho_n(t,x)}{L}\Big) \phi\Big(\frac{\rho_n(t,y)}{L}\Big) \leq 
2^\alpha \frac{\rho_n(t,x)^\alpha+\rho_n(t,y)^\alpha}{L^\alpha},
$$
since the left hand side vanishes when $\rho_n(t,x)\leq L/2$ and $\rho_n(t,y)\leq L/2$. Therefore, for the same $\alpha>0$ upon using the Cauchy-Schwarz inequality, the uniform bounds on $\calD(\rho_n \vu\textcolor{magenta}{_n})$ in $L^2([0,T]\times\RR^d)$ (see Corollary \ref{cor:D}) and on $\rho_n$ in $L^\infty(0,T;L^{q}(\RR^d))$ for $q\in (1,\gamma_n)$, we obtain
\eq{\label{ineq11}
&-\iintTO{K_h(x-y) (\calD(\rho_n \vu_n)(x)-\calD(\rho_n \vu_n)(y)) \rho_n(x)  \lr{\vr_n(x)-\vr_n(y)} w(x)} \\
&\leq C \|K_h\|_{L^1} L^{-\alpha}\\
&\quad+\int_0^T\!\!\!\int_{\RR^{2d}}K_h(x-y) (\calD(\rho_n \vu_n)(x)-\calD(\rho_n \vu_n)(y)) \\
&\hspace{2cm}\times\lr{\vr_n(y)-\vr_n(x)} \rho_n(x)  w(x)\phi\Big(\frac{\rho_n(t,x)}{L}\Big)\phi\Big(\frac{\rho_n(t,y)}{L}\Big)\,\dx\,\dy\,\dt.
}
Then, we may apply Lemma \ref{lem83} with the function 
\begin{equation}\label{eq:Phi}
\Phi(t,x,y) = \lr{\vr_n(y)-\vr_n(x)}\rho_n(x)  w(x)\phi\Big(\frac{\rho_n(t,x)}{L}\Big)\phi\Big(\frac{\rho_n(t,y)}{L}\Big),
\end{equation}
By definition of the truncation $\phi$, we have that $\|\Phi\|_{L^\infty} \leq C L^2$. 
For the control on the time derivative of $\Phi$, we notice that $\Phi$ is a
combination of functions $\rho_n$ and $w$ which satisfy 
a transport equation with the same velocity field, but different right hand sides. 
Then,
\eqh{
&\pa_t \Phi + \dv_x(\vu_n(x)\Phi) + \dv_y(\vu_n(y)\Phi) \\
&= f_1 \dv_x \vu_n(x) + f_2 \dv_x \vu_n(y) + f_3 B_n(x) + f_4 B_n(y)+f_5\rho_n(x)G(p_n(x))+f_6\rho_n(y)G(p_n(y)),
}
where $B_n$ is defined in \eqref{eq:weight} and $G(p_n)$ is defined in \eqref{hypG}.
Every function $f_i$ contain as a factor $\phi(\rho_n/L)$ or a derivative of $\phi$. Then $\|f_i\|_{L^\infty}\leq C L^2$ for $i=1,\ldots,4$.
We deduce that the constant $C_\Phi$ in Lemma \ref{lem83} is bounded by $C L^2$.
Thus,
\eqh{
&-\iintTO{K_h(x-y) (\calD(\rho_n \vu_n)(x)-\calD(\rho_n \vu_n)(y)) \lr{\vr_n(x)-\vr_n(y)}\rho_n(x) w(x)} \\
&\leq C \|K_h\|_{L^1} (h^\theta L^2 + L^{-\alpha}).
}
Optimizing in $L$, i.e. choosing $L=h^{-\theta/(\alpha+2)}$, we deduce that
there exists $\theta_0>0$ such that  
\eq{\label{ineq2}
-\iintTO{ \overline{K_h}(x-y) (\calD(\rho_n \vu_n)(x)-\calD(\rho_n \vu_n)(y)) \lr{\vr_n(x)-\vr_n(y)}\rho_n(x) w(x)} \leq C h^{\theta_0}.
}
Finally, integrating in time \eqref{estimdtR0} and inserting \eqref{ineq:lem63}
and \eqref{ineq2}, we obtain for all $t\in[0,T]$
\begin{align}
  e^{-G_0P_Mt} \calR_{h_0}(t) \leq \calR_{h_0}(0) + C_T\left(|\log h_0|^{1/2}+\int_{h_0}^1 h^{\theta_0}\,\frac{\dh}{h} \right).
\label{boundRh0}
\end{align}

\subsection{Removing the weights and compactness argument}\label{sec:remove}

Let $\eta<1$. We define $\omega_\eta = \{x: w \leq \eta\}$ and denote by 
$\omega_\eta^c$ its complementary.
We have 
\eq{
&\iintO{\calK_{h_0}(x-y)\lr{\rho_n(x)-\rho_n(y)}^2}\\
&= \int_{h_0}^1 \iintO{\overline{K_h}(x-y) \lr{\rho_n(x)-\rho_n(y)}^2}\frac{\dh}{h} = I_1 + I_2, \label{eqI12}
}
with
\begin{align*}
I_1 & = \int_{h_0}^1 \int_{\{x\in\omega_\eta^c\}\cup\{y\in\omega_\eta^c\}} \overline{K_h}(x-y) \lr{\rho_n(x)-\rho_n(y)}^2\,\dx\,\dy\,\frac{\dh}{h} \leq \frac{2}{\eta} \calR_{h_0},
\end{align*}
and
\begin{align*}
I_2 & =\int_{h_0}^1 \int_{\{x\in\omega_\eta\}\cap\{y\in\omega_\eta\}} \overline{K_h}(x-y) \lr{\rho_n(x)-\rho_n(y)}^2\,\dx\,\dy\,\frac{\dh}{h}  \\
&\leq C\int_{h_0}^1 \int_{\{x\in\omega_\eta\}\cap\{y\in\omega_\eta\}} \overline{K_h}(x-y) \rho^2_n(x)\,\dx\,\dy\,\frac{\dh}{h}  \\
&\leq C\int_{h_0}^1\int_{\RR^d}\overline{K_h}(z)  \, {\rm d}z\int_{\{x\in\omega_\eta\}}\rho^2_n(x)\,\dx\,\frac{\dh}{h}  \\
&\leq C\int_{h_0}^1\int_{\{x\in\omega_\eta\}}\rho^2_n(x)\,\dx\,\frac{\dh}{h}  \\
&\leq C|\log h_0|\int_{\{x\in\omega_\eta\}}\rho^2_n(x)\,\dx 
\end{align*}
where we used the symmetry of $K_h$ and the fact that $\|\overline{K_h}\|_{L^1}=1$. To treat the last integral we recall an interpolation inequality
\eqh{\|\vr_n\|_{L^2(\Omega)}\leq \|\vr_n\|_{L^1(\Omega)}^\tau\|\vr_n\|_{L^q(\Omega)}^{1-\tau}\leq C\|\vr_n\|_{L^1(\Omega)}^\tau,}
for $\vr_n\in L^\infty(0,T; L^q(\RR^d))$, where $\tau=\frac{q-2}{2(q-1)}$. Therefore
\eqh{
I_2\leq C|\log h_0|\lr{\int_{\{x\in\omega_\eta\}}\rho_n(x)\,\dx}^{2\tau}\leq C|\log h_0|\lr{\int_{\RR^d}\rho_n(x)\frac{|\log w(x)|}{|\log \eta|}\,\dx}^{2\tau}\leq \frac{C|\log h_0|}{|\log \eta|^{2\tau}},
}
since for $\eta<1$, $|\log w(x)|\geq |\log \eta|$ for all $x\in\omega_\eta$, and the last inequality follows by \eqref{boundlogw}.
Inserting these estimates on $I_1$ and $I_2$ into \eqref{eqI12}, we arrive at
\begin{equation}\label{es:Khsansw}
\int_{\RR^{2d}} \calK_{h_0}(x-y)\lr{\rho_n(x)-\rho_n(y)}^2\,\dx\,\dy 
\leq \frac{2}{\eta} \calR_{h_0} + \frac{C|\log h_0|}{|\log \eta|^{2\tau}}.
\end{equation}
Finally, from \eqref{boundRh0}, we deduce
$$
\int_{\RR^{2d}} \calK_{h_0}(x-y)\lr{\rho_n(x)-\rho_n(y)}^2\,\dx\,\dy 
\leq \frac{2}{\eta} \left(\calR_{h_0}(0) + C_T \left(|\log h_0|^{1/2} + 1-h_0^{\theta_0}\right)\right) 
+ \frac{C|\log h_0|}{|\log \eta|^{2\tau}}.
$$
Since we have $\|\calK_{h_0}\|_{L^1} \sim |\log h_0|$, we obtain
\begin{equation}\label{eqII}
\int_{\RR^{2d}} \overline{\calK_{h_0}}(x-y)\lr{\rho_n(x)-\rho_n(y)}^2\,\dx\,\dy 
\leq \frac{C_T}{\eta} \left(\frac{\calR_{h_0}(0)+ 1-h_0^{\theta_0}}{|\log h_0|}
+ |\log h_0|^{-1/2}\right) 
+ \frac{C}{|\log \eta|^{2\tau}}.
\end{equation}
Note that $2\tau< 1$, choosing $\eta = |\log h_0|^{-1/4}$, $\eta \to 0$ when $h_0\to 0$.
Then 
\eqh{
&\int_{\RR^{2d}} \overline{\calK_{h_0}}(x-y)\lr{\rho_n(x)-\rho_n(y)}^2\,\dx\,\dy \\
&\leq C_T \left( |\log h_0|^{-1/4}\left(\frac{\calR_{h_0}(0)+ 1-h_0^{\theta_0}}{|\log h_0|^{1/2}} + 1\right) 
+ \frac{C}{|\log |\log h_0||^{2\tau}} \right).
}
Finally, we obtain the compactness of the sequence $\{\rho_n\}_n$, as stated in 
Proposition \ref{prop:rho}, by applying the compactness criterion in
Lemma \ref{lem:compact}.
Indeed the estimate on the time derivative is a direct consequence of the
conservation equation \eqref{eq:NS1} and of the energy estimate in Lemma \ref{lem:nrj}.

\section{Limiting system}\label{sec:limit}
This section is dedicated to the limit passage $n\to \infty$ in the definition of the weak solutions to the approximate system (Definition \ref{weaksol}). We will first gather together all the uniform estimates for the sequence of solutions $\{\vr_n,\vu_n,\vr_n^{\gamma_n}\}_{n=1}^\infty$ and pass to the limit in the continuity and the momentum equation. Then we prove the complementary relation \eqref{eqlim4}. Finally, we also prove the complementary relation $\dv\vu=G(p_\infty)$ on the set $\{\rho_\infty=1\}$.

\subsection{Convergence in the continuity and the momentum equations}
{Following the estimates of Section \ref{sec:estimates} and the compactness result in Section \ref{sec:compactness}, there exists $(\rho_\infty,p_\infty,u_\infty)$ such that, for $n\to +\infty$, up to a subsequence, we have}
\begin{align}
& \rho_n \to \rho_\infty \quad \mbox{ strongly in } L^2_{loc}([0,T]\times \RR^d),
  \mbox{ (Proposition \ref{prop:rho})} \label{conv_1}\\
& \rho_n \mbox{ is uniformly bounded in } L^\infty(0,T;L^q(\RR^d)), q\geq 1, 
  \mbox{ (Lemma \ref{lem:nrj})} \label{conv_2}\\
&{ p_n = \rho_n^{\gamma_n} \rightharpoonup p_\infty \quad \mbox{ weakly in } L^2([0,T]\times\RR^d), }
  \mbox{ (Lemma \ref{lem:pL2})}  \label{conv_3}\\
& \vu_n \rightharpoonup \vu_\infty \quad \mbox{ weakly in } L^2(0,T;H^1_{loc}(\RR^d)),
  \mbox{ (Lemma \ref{lem:nrj})}  \label{conv_4}\\
& 0 \leq \rho_\infty \leq 1.  \mbox{ (Lemma \ref{lem:rhoinf1})} \label{conv_5}
\end{align}
From \eqref{conv_1} and \eqref{conv_2} by interpolation of the Lebesgue spaces, we deduce that 
\eq{ \rho_n \to \rho_\infty \quad \mbox{ strongly in } L^q_{loc}([0,T]\times \RR^d), \ q\geq1. \label{conv_6}
}
In addition, the time derivative of $\partial_t\rho_n$ can be expressed by means of equation \eqref{eq:NS1}, therefore the Arzel\'a-Ascoli theorem and the uniform estimate \eqref{conv_2} imply that 
\eq{ \rho_n \to \rho_\infty \quad \mbox{ in } C_{w}([0,T]; L^q( \RR^d)), \ q\geq1. \label{conv_7}
}
Moreover, uniformly with respect to $n$ we have 
\eq{ \|\sqrt{\rho_n}\vu_n\|_{L^\infty(0,T; L^2(\RR^d))}+\|\vu_n\|_{L^2(0,T; L^{\frac{2d}{d-2}}(\RR^d))}\leq C,}
and so, using also \eqref{conv_2} we get
\eq{&\|\rho_n\vu_n\|_{L^\infty(0,T; L^{q_0}(\RR^d))}+\|\rho_n\vu_n\|_{L^2(0,T; L^{q_1}(\RR^d))}\\
&\quad+
\|\rho_n|\vu_n|^2\|_{L^1(0,T; L^{q_2}(\RR^d))}
+\|\rho_n|\vu_n|^2\|_{L^2(0,T; L^{q_3}(\RR^d))}
\leq C,\label{un_ru}
}
for $1\leq q_0<2$, $1\leq q_1<\frac{2d}{d-2}$, $1\leq q_2<\frac{2d}{2(d-2)}$, $q_3<\frac{d}{d-2}$, and therefore
\begin{align}
& \rho_n\vu_n \rightharpoonup \overline{\rho\vu} \quad \mbox{ weakly* in } L^\infty(0,T;L^{q_0}(\RR^d)),\\
& \rho_n\vu_n \rightharpoonup \overline{\rho\vu} \quad \mbox{ weakly in } L^2(0,T;L^{q_1}(\RR^d)),\\
  & \rho_n\vu_n \otimes\vu_n \rightharpoonup \overline{\rho\vu\otimes\vu}  \quad \mbox{ weakly in } L^2(0,T;L^{q_3}(\RR^d).\label{conv_conv}
\end{align}
Combining \eqref{conv_6} with \eqref{conv_4} we check that 
\begin{align*}
\rho_n \vu_n \rightharpoonup \rho_\infty \vu_\infty  \quad \mbox{ weakly in } L^p_{loc}([0,T]\times \RR^d),\ 1\leq p<2,
\end{align*}
and therefore from the uniqueness of the weak limit $\overline{\rho\vu}=\rho_\infty \vu_\infty$, and also
\begin{align}
\rho_n \vu_n \rightharpoonup \rho_\infty \vu_\infty  \quad \mbox{ weakly* in } L^\infty(0,T;L^{q_0}_{loc}(\RR^d)).
\end{align}
Using the estimates of $p_n$, $\rho_n$ and $\vu_n$, we deduce that $\pa_t (\rho_n \vu_n) $, given by \eqref{eq:NS2}, is uniformly bounded in
$$L^2(0,T;W^{-1,q_3}(\RR^d))+L^2(0,T;W^{-1,2}(\RR^d))+L^\infty(0,T; L^q(\RR^d))+L^2(0,T;L^p(\RR^d)),$$
 for $1\leq p<2$. This estimate might be used to identify the limit in \eqref{conv_conv}. To this purpose, we recall the following compensated-compactness lemma, see \cite[Lemma 3.3]{LM}.
\begin{lemma}
Let $T>0$.
Let $(g_n)_n$ and $(f_n)_n$ be two sequences converging weakly towards $g$ and $f$, respectively in $L^{p_1}(0,T;L^{p_2}(\RR^d))$ and $L^{q_1}(0,T;L^{q_2}(\RR^d))$,
where $1\leq p_1,p_2\leq +\infty$, $\frac{1}{p_1}+\frac{1}{q_1}=\frac{1}{p_2}+\frac{1}{q_2}=1$.
Let us assume in addition that 
\begin{align*}
&\pa_t g_n \mbox{ is bounded in } \calM(0,T;W^{-m,1}(\RR^d)) \mbox{ for some }
m\geq 0 \mbox{ independent of }n;  \\
& \|f_n\|_{L^1(0,T;H^s(\RR^d))} \mbox{ is bounded for some } s>0.
\end{align*}  
Then $f_ng_n$ converges to $fg$ weakly in $\calD'([0,T]\times\RR^d)$.
\end{lemma}
Taking $g_n=\rho_n\vu_n$ and $f_n=\vu_n$ in this lemma, 
we justify that \eqref{conv_conv} is in fact
\eq{\label{conv_adv}
\rho_n\vu_n \otimes\vu_n \rightharpoonup \rho_\infty \vu_\infty \otimes \vu_\infty  \quad \mbox{ weakly in } L^2(0,T;L^{q_3}_{loc}(\RR^d)).
}
{
The last task is to pass to the limit in the production term of the momentum equation $\vr_n G(p_n)\vu_n$. To this purpose we first note that this sequence is weakly  convergent in $L^p((0,T)\times\RR^d)$ for some $p>1$ to a limit denoted by $\overline{\vr_\infty G(p_\infty)\vu_\infty}$. To identify this limit, we will use \eqref{conv_3} and the strong convergence of the sequence $\{\vr_n\vu_n\}_{n=1}^\infty$. To deduce the latter, we first note that \eqref{conv_1} and \eqref{conv_4} imply that
$\sqrt{\vr_n}\vu_n\rightharpoonup\sqrt{\vr_\infty}\vu_\infty$ weakly in $L^2(0,T; L^q_{loc}(\RR^d))$ for $q<\frac{2d}{d-2}$. Next, as in \eqref{conv_adv} we show that for any compact set $K\subset \RR^d$ we have
\eqh{\| \sqrt{\vr_n}\vu_n\|^2_{L^2(0,T; L^2_{loc}(\RR^d))}=\intT{\!\!\!\int_{K}\rho_n|\vu_n|^2\,\dx} \to \intT{\!\!\!\int_{K}\rho_\infty |\vu_\infty|^2\,\dx}
=\| \sqrt{\vr_\infty}\vu_\infty\|^2_{L^2(0,T;L^2_{loc}(\RR^d))}.}
The weak convergence of $\sqrt{\vr_n}\vu_n$ and the convergence of the $L^2$-norm implies that $\sqrt{\vr_n}\vu_n$ converges to $\sqrt{\vr_\infty}\vu_\infty$ strongly in ${L^2(0,T; L^2_{loc}(\RR^d))}$. From \eqref{conv_6} and from the uniform bounds on $\vr_n\vu_n$ in \eqref{un_ru} it then follows that ${\vr_n}\vu_n$ converges to  $\vr_\infty\vu_\infty$ strongly in ${L^2(0,T; L^2_{loc}(\RR^d))}$.}

{
This concludes the proof of the passage to the limit in the continuity and in the momentum equations leading to the weak solution from Definition \ref{weaksollim}.}

\subsection{Passage to the limit in the congestion relation}
Here we follow a similar argument from \cite{LM}. In order to recover relation \eqref{eqlim4} we first see that for any $\delta>0$, there exists $n_0$ sufficiently large such that for $n\geq n_0$ we have
$$\vr_n^{\gamma_n+1}\geq \vr_n^{\gamma_n}-\delta.$$
Thus, passing with $n$ to the limit we obtain
$$\overline{\vr_n^{\gamma_n+1}}\geq p_\infty-\delta. $$
The limit on the left hand side can be immediately identified with $\rho_\infty p_\infty$, due to the strong convergence of $\vr_n$ and weak convergence of $p_n$. Therefore, letting $\delta\to0$, we get
$$\rho_\infty p_\infty\geq p_\infty. $$
Note however, that due to \eqref{conv_5}, $\rho_\infty\leq 1$, therefore $\rho_\infty p_\infty\leq p_\infty$, which implies that $\rho_\infty p_\infty= p_\infty$.

\subsection{Consistency relation}

In the following lemma, we show that conditions \eqref{sys:compressible} and \eqref{sys:incompressible} are compatible. This is provided by the equivalency of the following conditions.
\begin{lemma}\label{lem:compl}
Let $\vu\in L^2(0,T;H^1_{loc}(\RR^d))$, $\rho\in L^2_{loc}([0,T]\times\RR^d)$, and $G(p)\in L^2_{loc}([0,T]\times\RR^d)$, where $\rho\geq 0$ a.e. in $(0,T)\times\RR^d$
satisfy the transport equation
\begin{equation}\label{eqrhofin}
\pa_t \rho + \dv(\rho \vu) = \rho G(p) \quad \mbox{ in } (0,T)\times \RR^d,
\quad \rho(t=0)=\rho^0.
\end{equation}
Then the following two assertions are equivalent
\begin{itemize}
\item[(i)] $\dv \vu=G(p)$ a.e. on $\{\rho\geq 1\}$ and $0\leq \rho^0 \leq 1$,
\item[(ii)] $0\leq \rho \leq 1$, for any $t\in[0,T]$.
\end{itemize}
\end{lemma}

\begin{proof} We follow the idea from \cite[Lemma 2.1]{LM}. We first prove the implication $(i)$ $\Rightarrow$ $(ii)$.
From the renormalization property,
we have that for any $C^1$ function $\beta$ from $\RR$ to $\RR$ such that
$|\beta(t)|\leq C(1+t)$,
\begin{equation}\label{eq:betarho}
\pa_t \beta(\rho) + \dv(\beta(\rho) \vu) = (\beta(\rho)-\rho \beta'(\rho)) \dv \vu
+ \rho \beta'(\rho) G(p).
\end{equation}
We choose for $\beta$ the function
$$
\beta(y) = \left\{ \begin{array}{ll}
0, \qquad  & \mbox{ if } y\leq 0,  \\
y, \qquad  & \mbox{ if } y\in (0,1),  \\
1, \qquad  & \mbox{ if } y\geq 1.  \\
\end{array}\right.
$$
Then we get (after regularization and passing to the limit for the rigorous 
justification):
$$
\partial_t \beta(\rho) + \dv(\beta(\rho)\vu) = \mathbf{1}_{\{\rho\geq 1\}} \dv \vu
+ \mathbf{1}_{\rho\in (0,1)} \rho G(p).
$$
Denoting $\sigma=\beta(\rho)-\rho$ and subtracting from the latter equation  \eqref{eqrhofin}, we obtain
$$
\pa_t \sigma + \dv(\sigma\vu) = \mathbf{1}_{\{\rho\geq 1\}} G(p)(1 - \rho),
$$
where we used the assumption $\dv \vu = G(p)$ on $\{\rho\geq 1\}$.
Moreover, thanks to our choice of function $\beta$, we have $\sigma=\beta(\rho)-\rho=(1-\rho)\mathbf{1}_{\{\rho\geq 1\}}$.
Therefore, we arrive at
$$
\pa_t \sigma + \dv(\sigma\vu) = \sigma G(p).
$$
It is classical to deduce that $|\sigma|$ satisfies the same equation.
Integrating it over $\RR^d$, we obtain
$$
\Dt\int_{\RR^d} |\sigma(t)|\,dx \leq G_0P_M \int_{\RR^d} |\sigma(t)|\,\dx.
$$
Note that $\sigma(0)=0$, since by $(i)$ $0\leq \rho^0\leq 1$. Therefore, using the Gronwall lemma we conclude that $0=|\sigma|=(1-\rho)\mathbf{1}_{\rho\geq 1}$ which implies $(ii)$.

For the reverse implication, $(ii)$ $\Rightarrow$ $(i)$ we proceed as follows. Since $\rho$
is bounded, equation \eqref{eq:betarho} holds for any $C^1$ function $\beta$.
In particular, for $\beta(\rho)=\rho^k$, for any integer $k$, we get
$$
\pa_t \rho^k + \dv(\rho^k \vu) = \left[(1-k) \dv \vu + k G(p)\right]\rho^k. 
$$
By $(ii)$ $0\leq \rho^k \leq 1$, thus $\pa_t \rho^k$ is bounded in $W^{-1,\infty}((0,T)\times\RR^d)$.
Since $|\rho^k \vu|\leq |\rho \vu|$, we deduce that $\dv(\rho^k \vu)$
is bounded in $L^\infty(0,T;H^{-1}_{loc}(\RR^d))$, and because $|\rho^k \dv \vu| \leq |\dv \vu|$, $\rho^k\dv\vu$ is bounded in $L^2_{loc}([0,T]\times\RR^d)$. This means that $k\rho^k (G(p)-\dv\vu)$ is a distribution bounded uniformly with respect to $k$.
We deduce that we can pass into the limit $k\to\infty$ we therefore obtain
$$
\rho^k(G(p) - \dv \vu) \rightharpoonup 0, \qquad \mbox{ in the sense of distributions.}
$$
Moreover, we have that $\rho^k\to \mathbf{1}_{\rho=1}$ a.e., it implies that
$$
\rho^k(G(p) - \dv \vu) \to (G(p) - \dv \vu)\mathbf{1}_{\{\rho=1\}}\quad \mbox{ a.e. in } (0,T)\times\RR^d.
$$ 
Comparing the limits we obtain $G(p)=\dv \vu$ a.e. on $\{\rho=1\}$, which implies $(i)$. 
\end{proof}

\section{About existence of solutions}\label{sec:existence}
In this section we explain the main steps leading to the construction of the weak solutions from Definition \ref{weaksol}. We will explain how this solution can be obtained by chain of approximations of system \eqref{sys:main}, inluding parabolic regularization of the continuity equation and the Faedo-Galerkin approximation of the momentum equation.

\subsection{Existence of solutions to system with additional dissipation}
The weak solution from Definition \ref{weaksol} will be obtained as a limit $(\vr,\vu)$ as $\eps\to0^+$ of the weak solutions $(\vr_\ep,\vu_\ep)$ to the following system with artificial viscosity
\begin{subequations}\label{sys:eps}
\begin{align}
\label{eq:rhoeps}
&\pa_t \rho_\ep + \dv(\rho_\ep \vu_\ep) = \rho_\ep G(p_\ep) + \eps \Delta \rho_\ep,  \\
\label{eq:rhoueps}
&\pa_t (\rho_\ep \vu_\ep) + \dv ( \rho_\ep \vu_\ep \otimes \vu_\ep) + \nabla p(\vr_\ep) - \mu \Delta \vu_\ep - \xi \nabla \dv \vu_\ep =
 \rho_\ep \vu_\ep G(p(\vr_\ep)) - \eps \nabla \rho_\ep \cdot \nabla \vu_\ep.
\end{align}
\end{subequations}
The existence of solutions to system \eqref{sys:eps} is guaranteed by the following theorem.
\begin{theorem}\label{Th:ep}
Let $T>0$, and $\gamma\geq2$, $\eps>0$ be fixed. Let the initial conditions be given by \eqref{hypini}. Then, there exists a weak solution $(\vr_\ep,\vu_\ep)$ to the system \eqref{sys:eps} with the boundary conditions \eqref{hypbound}, the pressure given by \eqref{eq:p} and $G$ given by \eqref{hypG}. More precisely, the following norms on $\vr_\ep$ and $\vu_\ep$ are bounded uniformly in $\ep$:
\begin{subequations}\label{estim:eps}
\begin{align}
&\|\vr_\eps\|_{L^\infty(0,T; L^\gamma(\RR^d))}+\|\vr_\eps\|_{L^{2\gamma}((0,T)\times \RR^d)}\leq C,\label{es:rhoep}\\
&\sqrt{\eps}\|\Grad\vr_\eps\|_{L^{2}((0,T)\times \RR^d)}+\sqrt{\eps}\|\Grad\vr_\eps^{\frac{\gamma}{2}}\|_{L^{2}((0,T)\times \RR^d)}\leq C,\label{es:gradrhoep}\\
 &\|\sqrt{\rho_\eps}\vu_\eps\|_{L^\infty(0,T; L^2(\RR^d))}+\|\vu_\eps\|_{L^2(0,T; H_{loc}^{1}(\RR^d))}\leq C\label{es:uep},
\end{align}
\end{subequations}
and $\vr_\ep$, $\vu_\ep$ satisfy the equations \eqref{sys:eps} in the sense of distributions. 
\end{theorem}
\begin{proof}
The solution to system \eqref{sys:eps} can be constructed using the  invading domains approach described in \cite[Chapter 7]{NS}. This means to find the solution to \eqref{sys:eps} on a bounded domain $\Omega_R=B(0,R)$ first and then to let $R\to \infty$. To prove that \eqref{sys:eps} has a weak solution on $\Omega_R$, we need to supplement the system  with Dirichlet boundary conditions for $\vu_\ep$ and the zero Neumann boundary condition for $\vr_\ep$. The weak solutions to such problem can be constructed by the Faedo-Galerkin discretization of the momentum equation \eqref{eq:rhoueps} and the fixed point argument.  The details of the last two steps are only slight modification of the procedure from \cite{NS} as all the additional terms related to $G(p_\ep)$ are of lower order and the basic a-priori estimates are still valid.

Saying this, let us recall that at the level of Faedo-Galerkin approximation $\vu_\ep$ is a suitable test function for the momentum equation and the continuity equation is satisfied pointwisely. Therefore, the energy estimate can be justified rigorously and it implies the following uniform in $\ep$ bounds.
\begin{lemma}\label{lem:nrj2}
Under assumptions \eqref{hypini} and \eqref{hypG}, let $T>0$ and $\eps>0$ be fixed,
then there exists a nonnegative constant $C$ (uniform in $\eps$) such that
the weak solution $(\rho_\eps,\vu_\eps)$ of Theorem \ref{Th:ep} satisfies,
for all $t\in [0,T]$,
\eq{\label{est:uep}
\calE_\ep(t) + \int_0^t \calJ_\ep(s)\,{\rm{d}}s 
+ \eps \gamma \int_0^t\int_{\RR^d} \rho_\ep^{\gamma-2} |\nabla \rho_\ep|^2 \dx\, {\rm d}s
\leq (\calE_\ep(0) + Ct) e^{G_0P_M t},
}
with $\calE_\ep(t) $ and $\calJ_\ep(t) $ defined in \eqref{energy}, \eqref{dissip}.
\end{lemma}
\begin{proof}
The proof of this fact follows exactly the proof of the energy estimate \eqref{est:energy}. The extra term in the momentum form $\eps \nabla \rho_\ep \cdot \nabla \vu_\ep$ allows to cancel the extra term coming from multiplication of the continuity equation by $\frac{|\vu_\ep|^2}{2}$.
\end{proof}

We can also easily check that the estimate of the pressure from Lemma \ref{lem:pL2} is valid. 
Indeed, multiplying \eqref{eq:rhoeps} by $\gamma \rho_\eps^{\gamma-1}$, we deduce the equation for the pressure 
\begin{equation}\label{edp:peps}
\pa_t p_\eps + \gamma p_\eps \dv \vu_\eps + \vu_\eps\cdot \nabla p_\eps =
\gamma p_\eps G(p_\eps) + \eps \Delta p_\eps - \eps \gamma(\gamma-1) \rho_\eps^{\gamma-2} |\nabla \rho_\eps|^2.
\end{equation}

\begin{lemma}\label{lem:eseps}
Let $\gamma\geq 2$ and let the initial conditions satisfy \eqref{hypini}. Then there exists a positive constant $C$ such that uniformly with respect to $\ep$ 
we have
\eq{\label{est:pep}
\|\vr_\ep^\gamma\|_{L^\infty(0,T; L^1(\RR^d))}+\|\vr_\ep^\gamma\|^2_{L^2((0,T)\times\RR^d)}+\ep\|\sqrt{p''(\vr_\ep)}\Grad\vr_\ep\|_{L^2((0,T)\times\RR^d)}^2\leq C.}
Moreover, uniformly with respect to $\ep$ we have
\eq{\label{eps_rho}
\|\vr_\ep\|_{L^\infty(0,T; L^1(\RR^d))}+\sqrt{\ep}\|\Grad\vr_\ep\|_{L^2((0,T)\times\RR^d)}\leq C.
}
\end{lemma}
\begin{proof}
The proof of the first estimate \eqref{est:pep} follows directly by an integration of \eqref{edp:peps} over $\Omega_R$ and by letting $R\to\infty$.
The proof of the first part in estimate \eqref{eps_rho} follows directly by integration of \eqref{eq:rhoeps} over the space. To prove the second bound in \eqref{eps_rho}, we multiply the continuity equation \eqref{eq:rhoeps} by $\rho_\eps$. Integrating by parts we obtain
\eqh{
&\frac{1}{2}\|\vr_\ep(T)\|_{L^2(\RR^d)}^2+\ep\intT{\|\Grad\vr_\ep\|^2_{L^2(\RR^d)}}+G_0\intT{\|\vr_\ep\|^{\gamma+2}_{L^{\gamma+2}}}\\
&=\frac{1}{2}\|\vr_\ep(0)\|_{L^2(\RR^d)}^2+G_0P_M\intTO{\vr_\ep^2}-\frac{1}{2}\intTO{\vr_\ep^2\Div\vu_\ep}.
}
The last two terms can be bounded using \eqref{est:uep} and \eqref{est:pep}, on account of the fact that $\gamma\geq2$.
\end{proof}

\medskip
With these estimates at hand, the proof of Theorem \ref{Th:ep} is complete.
\end{proof}

\subsection{Passage to the limit $\ep\to 0$}

Existence of weak solutions to our initial system \eqref{sys:main} is then obtained by passing to the limit $\ep\to 0$.
\begin{theorem}\label{Th:exist}
Let $T>0$, and $\gamma$ large enough be fixed. Let the initial conditions be given by \eqref{hypini}. Then, there exists a weak solution $(\vr,\vu)$ to the system \eqref{sys:main} in the sense of Definition \ref{weaksol}, with the boundary conditions \eqref{hypbound}, the pressure given by \eqref{eq:p} and $G$ given by \eqref{hypG}.
\end{theorem}
\begin{proof}
In order to perform the passage to the limit $\ep \to 0$ in the equations of system \eqref{sys:eps} first note that all the $\eps$-related terms converge to $0$ in the distributional formulation of the system. More precisely, from \eqref{es:gradrhoep} and \eqref{es:uep} it follows that
\eqh{&\ep\Grad\vr_\ep\to 0\quad \mbox{strongly in } L^2((0,T)\times\Omega),\\
&\ep\Grad\vr_\ep\cdot\Grad\vu_\ep\to 0\quad \mbox{strongly in } L^1((0,T)\times\Omega).
}
To pass to the limit in the rest of the terms of system \eqref{sys:eps}, one needs to combine the arguments from Section \ref{sec:limit} with the compactness of the sequence approximating the density $\{\vr_\ep\}_{\ep>0}$.
Note, that in Section \ref{sec:limit} we were using the property \eqref{conv_2} which is not available for $\gamma$ fixed. However, taking $\gamma$ sufficiently large one can still repeat all of the steps. The important changes concern solely the 
compactness argument for the sequence $\{\vr_\ep\}_{\ep>0}$.
Then in the rest of the proof, we only explain how to modify the method presented in Section \ref{sec:compactness} to handle the extra $\eps$-related terms and get compactness for the sequence $\{\vr_\ep\}_{\ep>0}$.

\subsubsection{Modified definition of the weights}

We first modify the weight by replacing the equation \eqref{eq:weight} into
\begin{equation}\label{eq:weighteps}
\pa_t w_\eps + \vu_\eps \cdot \nabla w_\eps = -\lambda B_\eps w_\eps+\eps\Delta w_\eps, 
\qquad B_\eps = M|\nabla \vu_\eps|,
\end{equation}
complemented with the initial data $w_\eps(t=0)=1$.
Here $\lambda$ is some nonnegative constant which will be fixed later on.
We establish a similar property as Proposition \ref{prop:w} for this weight.
\begin{lemma}\label{lem:weps}
Let us assume that $\vu_\eps$ is given and uniformly bounded with respect to $\eps$ in $L^2_{loc}([0,T]\times\RR^d)\cap L^\infty(0,T;H^1(\RR^d))$.
Then, there exists a unique solution to \eqref{eq:weighteps}. Moreover, we have
\begin{itemize}
\item[(i)] For any $(t,x)\in(0,T)\times\RR^d$, $0\leq w_\eps(t,x) \leq 1$.
\item[(ii)] If we assume moreover that the pair $(\rho_\eps,\vu_\eps)$ solves \eqref{eq:rhoeps} and $\vr_\eps$
is uniformly bounded in $L^\infty([0,T]; L^1\cap L^\gamma(\RR^d))$ for $\gamma\geq 2$, then there exists $C\geq 0$, such that
$\intO{ \rho_\eps |\log w_\eps| } \leq C.$
\end{itemize}
\end{lemma}
\begin{proof}
$(i)$ Since \eqref{eq:weighteps} is a parabolic equation with $B_\eps$ nonnegative and with initial data $w_\eps(t=0)=1$, we have that $0\leq w_\eps(t,x) \leq 1$. 

$(ii)$ Since $w_\eps \leq 1$, $|\log w_\eps| = -\log w_\eps$, then we have from \eqref{eq:rhoeps}, \eqref{eq:weighteps},
\begin{align*}
\pa_t (\rho_\eps |\log w_\eps|) + \dv (\rho_\eps \vu_\eps |\log w_\eps|) =
&\ \lambda B_\eps \rho_\eps + \rho_\eps |\log w_\eps| G(p_\eps) 
+ \eps \Delta (\rho_\eps |\log w_\eps|)  \\
& - 2 \eps \nabla \rho_\eps \cdot \nabla |\log w_\eps| - \eps \rho_\eps |\nabla \log w_\eps |^2.
\end{align*}
Integrating with respect to space, and using \eqref{hypG}, we obtain
\begin{align}
\frac{d}{dt} \int_{\RR^d} \rho_\eps |\log w_\eps| \,\dx \leq &
\int_{\RR^d} \lambda B_\eps \rho_\eps \,\dx + G_0 P_M \int_{\RR^d} \rho_\eps |\log w_\eps| \,\dx  \nonumber \\
& - \eps \int_{\RR^d} \rho_\eps |\nabla \log w_\eps|^2 \,\dx - 2\eps \int_{\RR^d} \nabla \rho_\eps \cdot \nabla |\log w_\eps| \,\dx.
\label{eq:rholnw}
\end{align}
From $|\log w_\eps| = -\log w_\eps$, the Cauchy-Schwarz and the Young inequalities, we have 
\begin{equation}\label{eq:rholnrho}
2 \ep\left|\int_{\RR^d} \nabla \rho_\eps \cdot \nabla |\log w_\eps| \,\dx\right|
\leq \frac \ep2 \int_{\RR^d} \rho_\eps |\nabla \log w_\eps|^2 \,\dx
+\ep \int_{\RR^d} \frac{|\nabla \rho_\eps|^2}{\rho_\eps}\,\dx.
\end{equation}
Moreover, from \eqref{eq:rhoeps}, we deduce
\begin{align*}
\frac{d}{dt} \int_{\RR^d} \rho_\eps \log \rho_\eps \,\dx 
+ \int_{\RR^d} \rho_\eps \dv \vu_\eps \,\dx 
& = \int_{\RR^d} \rho_\eps(\log \rho_\eps + 1) G(p_\eps) \,\dx 
- \eps \int_{\RR^d} \frac{|\nabla \rho_\eps|^2}{\rho_\eps} \,\dx  \\
& \leq G_0(P_M+1) \int_{\RR^d} \rho_\eps(|\log \rho_\eps| + 1) \,\dx 
- \eps \int_{\RR^d} \frac{|\nabla \rho_\eps|^2}{\rho_\eps} \,\dx.
\end{align*}
Since $\rho_\eps$ is uniformly bounded in $L^\infty([0,T]; L^1(\RR^d)\cap L^\gamma(\RR^d))$ for $\gamma\geq 2$,
 then $\rho_\eps\log \rho_\eps$ is uniformly bounded in $L^\infty([0,T];L^1(\RR^d))$.
Moreover, $\dv \vu_\eps$ is uniformly bounded in $L^2([0,T]\times\RR^d)$, 
therefore, we deduce after an integration in time of the above inequality,
that there exists a nonnegative constant $C$ such that
\begin{equation}\label{es:nrr}
\eps \intTO{\frac{|\nabla \rho_\eps|^2}{\rho_\eps}} \leq  C.
\end{equation}
Integrating \eqref{eq:rholnw} with respect to time, inserting \eqref{eq:rholnrho} and \eqref{es:nrr}, we conclude the proof since $B_\eps$ and $\rho_\eps$ are uniformly bounded in $L^2([0,T]\times\RR^d)$.
\end{proof}

\subsubsection{Changes in the compactness argument}
To prove the local compactness of the sequence $\{\rho_\eps\}_{\eps>0}$, we adapt the argument of Section \ref{sec:compactrho}. We explain briefly the main change in the proof.
Starting from the transport equations \eqref{eq:rhoeps} satisfied by $\rho_\eps(x)$ and $\rho_\eps(y)$, making the difference and multiplying by $(\rho_\eps(x)-\rho_\eps(y))$, we deduce
\begin{align*}
&\frac{1}{2}\pa_t (\rho_\eps(x)-\rho_\eps(y))^2 + \frac{1}{2}\dv_x (\vu_\eps(x)\lr{\rho_\eps(x)-\rho_\eps(y)}^2) +
\frac{1}{2}\dv_y (\vu_\eps(y)\lr{\rho_\eps(x)-\rho_\eps(y)}^2) \\
=&- \frac 12 (\dv_x \vu_\eps(x)-\dv_y \vu_\eps(y))(\rho_\eps(x)+\rho_\eps(y)) \lr{\rho_\eps(x)-\rho_\eps(y)} \\
&+ \lr{\rho_\eps(x) G(p_\eps(x)) -\rho_\eps(y) G(p_\eps(y))} \lr{\rho_\eps(x)-\rho_\eps(y)}  \\
&+ \frac{\eps}{2} \Delta_{x,y} (\rho_\eps(x)-\rho_\eps(y))^2 - \eps |\nabla_{x,y}(\rho_\eps(x)-\rho_\eps(y))|^2.
\end{align*}
Following the reasoning of Section \ref{sec:compactrho}, we arrive at the analogue of \eqref{eq:drho1} with an extra term due to artificial viscosity
\eq{\label{eq:drho12}
&\frac{1}{2}\pa_t (\rho_\eps(x)-\rho_\eps(y))^2 + \frac{1}{2}\dv_x (\vu_\eps(x)\lr{\rho_\eps(x)-\rho_\eps(y)}^2) +
\frac{1}{2}\dv_y (\vu_\eps(y)\lr{\rho_\eps(x)-\rho_\eps(y)}^2) \\
\leq&- \frac 12 (\dv_x \vu_\eps(x)-\dv_y \vu_\eps(y))(\rho_\eps(x)+\rho_\eps(y)) \lr{\rho_\eps(x)-\rho_\eps(y)} \\
&+ G_0P_M\lr{\rho_\eps(x)-\rho_\eps(y)}^2 + \frac{\eps}{2} \Delta_{x,y} (\rho_\eps(x)-\rho_\eps(y))^2. 
}
Then, we introduce the regularization  of the weights $w_\eps$ satisfying \eqref{eq:weighteps}
$$
W_h(x,y) = \overline{K_h}*w_\eps(x) +  \overline{K_h}*w_\eps(y).
$$
We now take
$$
R(t) = \frac12\iintO{K_h(x-y) \lr{\rho_\eps(x)-\rho_\eps(y)}^2 W_h(x,y)},
$$
and
$$
\calR_{h_0}(t) = \frac12\int_{h_0}^1\iintO{\overline{K_{h}}(x-y) \lr{\rho_\eps(x)-\rho_\eps(y)}^2 W_h(x,y)} \frac{{\rm d}h}{h}
= \frac{1}{\|K_h\|_{L^1}} \int_{h_0}^1 R(t)\frac{{\rm d}h}{h}.
$$
Using \eqref{eq:drho12} and the symmetry of $K_h$, we deduce
\begin{equation}\label{eqdtReps}
\frac{d}{dt} R(t) \leq A_1 + A_2 + A_3 + A_4 + G_0P_M R(t),
\end{equation}
where
$$
A_1 = \frac12\iintO{\nabla K_h(x-y) (\vu_\eps(x)-\vu_\eps(y)) \lr{\rho_\eps(x)-\rho_\eps(y)}^2 W_h(x,y)},
$$
$$
A_2 =  \iintO{ K_h(x-y)\lr{\rho_\eps(x)-\rho_\eps(y)}^2 \overline{K_h}*(\pa_t w(y) + \vu_\eps(y)\cdot\nabla w_\eps(y) -\eps \Delta w_\eps(y))},
$$
$$
A_3 = -2 \iintO{ K_h(x-y)(\dv \vu_n(x) - \dv \vu_n(y)) \rho_n(x) \lr{\vr_n(x)-\vr_n(y)}\rho_n(x)\overline{K_h}*w_\eps(x) },
$$
$$
A_4 = \eps \iintO{ \lr{\Delta K_h(x-y) W_h(x,y) + K_h(x-y) \Delta W_h(x,y)} \lr{\vr_n(x)-\vr_n(y)}^2 }.
$$
Inequality \eqref{eqdtReps} is the equivalent to \eqref{eqdtR} derived in Section \ref{sec:compactrho} for no artificial viscosity case.

We estimate the new term $A_4$ by noticing that by definition of $K_h$ we have $\Delta K_h\leq \frac{C}{h^2} K_h$.
Then we may bound
\begin{equation}\label{es:A4}
A_4 \leq \frac{C\eps}{h^2} \|K_h\|_{L^1(\RR^d)}.
\end{equation}
The terms $A_1$ and $A_2$ may be estimated as before. For the term $A_3$, the estimate should be adapted 
since the relation \eqref{link:up} is not valid anymore. Indeed, there is an extra term
$$
(\mu+\xi) \dv \vu_\eps = p_\eps + \calD(\rho_\eps \vu_\eps) + F_\eps,
$$
where $F_\eps = \eps (-\Delta)^{-1} (\dv(\dv(\vu_\eps\otimes \nabla \rho_\eps))).$

Hence, we arrive at the following  equivalent of \eqref{estimdtR0},
\eq{\label{es:R0eps}
\Dt\calR_{h_0}(t) \leq & \ G(0) \calR_{h_0}(t) + C
\int_{h_0}^1 \int_{\RR^{d}} \overline{K_h}(z) \|D_{|z|}\vu_\ep(\cdot) - D_{|z|}\vu_\ep(\cdot+z)\|_{L^2(\RR^d)}\, {\rm d}z \,\frac{\dh}{h}  \\
& - \frac{2}{\mu+\xi} \int_{h_0}^1 \iintO{\overline{K_{h}}(x-y) \big(\calD(\rho_\eps \vu_\eps)(x)+F_\eps(x)-\calD(\rho_\eps \vu_\eps)(y)-F_\eps(y)\big) \\
& \hspace{5cm} \times \lr{\vr_\eps(x)-\vr_\eps(y)}\rho_\eps(x) \overline{K_h}*w_\eps(x)}\frac{\dh}{h}  \\
& + C \int_{h_0}^1 \eps\, \frac{\dh}{h^3}.
}
The second term on the right hand side may be controlled as before thanks to \eqref{ineq:lem63}.
To control the third term on the right hand side of \eqref{es:R0eps}, we truncate using the function 
$\phi$ as in Section \ref{sec:compactrho}. Since $\calD(\rho_\eps\vu_\eps)+F_\eps$ is uniformly bounded in
$L^2([0,T]\times \RR^d)$, we may write as before (see \eqref{ineq11}),
\eq{\label{es:next}
&-\iintTO{K_h(x-y) (\calD(\rho_\eps \vu_\eps)(x)+F_\eps(x)-\calD(\rho_\eps \vu_\eps)(y)-F_\eps(y))  \\
&\hspace{2cm}\times \rho_\eps(x)  \lr{\vr_\eps(x)-\vr_\eps(y)} \overline{K_h}*w_\ep(x)} \\
&\leq C \|K_h\|_{L^1} L^{-\alpha}\\
&\quad+\int_0^T\!\!\!\int_{\RR^{2d}}K_h(x-y) (\calD(\rho_\eps \vu_\eps)(x)-\calD(\rho_\eps \vu_\eps)(y)) \Phi_h(t,x,y)\,\dx\,\dy\,\dt \\
&\quad + \int_0^T\!\!\!\int_{\RR^{2d}}K_h(x-y) (F_\eps(x)-F_\eps(y)) \Phi_h(t,x,y)\,\dx\,\dy\,\dt
}
where the function $\Phi_h$ is defined, similarily as in \eqref{eq:Phi}, by
$$
\Phi_h(t,x,y) = \lr{\vr_\ep(y)-\vr_\ep(x)}\rho_\ep(x)  \overline{K_h}*w_\ep(x)\phi\Big(\frac{\rho_\ep(t,x)}{L}\Big)\phi\Big(\frac{\rho_\ep(t,y)}{L}\Big).
$$
By definition of the truncation $\phi$, we have that $\|\Phi_h\|_{L^\infty} \leq C L^2$. 
In particular, it allows us to use Lemma \ref{lem83} to bound the second term on the right hand side.  Here we actually use the extension of Lemma \ref{lem83} to the case when  $\vr_\ep$ satisfies the continuity equation with additional dissipation term \eqref{eq:rhoeps}. 
On account of Remark \ref{rem:Lemma8.3} and Lemma 8.3 in \cite{BJ} the resulting estimate is the same.
The last term in \eqref{es:next}, thanks to the truncation, may be bounded by
$$
2L^2 \int_0^T\!\!\!\int_{\RR^{2d}}K_h(x-y) |F_\eps(x)-F_\eps(y)|\,\dx\,\dy\,\dt.
$$
From Lemma \ref{lem:nrj2} and Lemma \ref{lem:eseps}, we deduce that the sequence $\{\frac{1}{\sqrt{\eps}} F_\eps\}_{\eps>0}$ is uniformly (with respect to $\eps$) bounded in $L^1_{loc}([0,T]\times\RR^d)$. Therefore the sequence $\{\eps^{-1/4} F_\eps\}_{\eps>0}$ converges to $0$ strongly, and therefore is compact in $L^1_{loc}([0,T]\times\RR^d)$. 
On account of Lemma \ref{lem:compact} it implies that for
$$
\epsilon_F^{}(h) :=  \frac{\eps^{-1/4}}{\|K_h\|_{L^1}}\int_0^T\!\!\!\int_{\RR^{2d}}K_h(x-y) |F_\eps(x)-F_\eps(y)|\,\dx\,\dy\,\dt,
$$
we have
\begin{equation}\label{limsupeF}
\underset{\eps> 0}{\lim\sup} \ \epsilon_F^{}(h) \to 0,
  \quad \mbox{ as } h\to 0.
\end{equation}
Thus integrating in time \eqref{es:R0eps}, using \eqref{ineq:lem63} and Lemma \ref{lem83}, we arrive at
$$
e^{-G_0 P_M t}\calR_{h_0}(t) \leq \calR_{h_0}(0) + C_T |\log h_0|^{1/2} + C_T \int_{h_0}^1 \left( L^{-\alpha} + h^{\theta} L^2 + 2 L^2 \eps^{1/4} \epsilon_F^{}(h)\right)\frac{\dh}{h}
+\frac{C\eps}{{h_0}^2}.
$$
Choosing $L=h^{-\theta/(\alpha+2)}$, we deduce that there exists $\theta_0=\frac{\alpha \theta}{\alpha+2}$ such that
\begin{equation}\label{estim:calReps}
e^{-G_0 P_M t} \calR_{h_0}(t) \leq \calR_{h_0}(0) + C_T \left(|\log h_0|^{1/2} + \int_{h_0}^1 h^{\theta_0}\frac{\dh}{h} 
+\eps^{1/4}\int_{h_0}^1 h^{-\alpha \theta_0/2}\epsilon_F^{}(h)\frac{\dh}{h} + \frac{\eps}{{h_0}^2} \right).
\end{equation}
This estimate is the equivalent to estimate \eqref{boundRh0}.

\subsubsection{Removing the weights}

The last step consists in removing the weight. 
Introducing $\omega_\eta = \{x : \overline{K_h}*w_\eps \leq \eta\}$, we use the same idea as in Section \ref{sec:remove} to remove the weight $w_\eps$, using Lemma \ref{lem:weps}, and arrive at a similar estimate as \eqref{es:Khsansw}
$$
\int_{\RR^{2d}} \calK_{h_0}(x-y)\lr{\rho_\eps(x)-\rho_\eps(y)}^2\,\dx\,\dy 
\leq \frac{2}{\eta} \calR_{h_0} + \frac{C|\log h_0|}{|\log \eta|^{2\tau}},
$$
for some $\tau<\frac12$.
Then, from \eqref{estim:calReps}, we deduce, by the same token as for \eqref{eqII}, that
\eq{\label{Kfinal}
\int_{\RR^{2d}} \overline{\calK_{h_0}}(x-y)\lr{\rho_\eps(x)-\rho_\eps(y)}^2\,\dx\,\dy 
\leq \  & \frac{C_T}{\eta} \left(\frac{\calR_{h_0}(0)+ 1-h_0^{\theta_0}}{|\log h_0|}
+ |\log h_0|^{-1/2}\right)  \\
& +\frac{\eps^{1/4}}{|\log h_0|}\int_{h_0}^1 h^{-\alpha \theta_0/2}\epsilon_F^{}(h)\frac{\dh}{h} + \frac{\eps}{{h_0}^2|\log h_0|}
 + \frac{C}{|\log \eta|^{2\tau}}.
}
Since, from \eqref{limsupeF}, we deduce that $\epsilon_F^{}$ is uniformly bounded with respect to $\eps$ 
for $\eps$ small enough, we obtain
$$
\lim_{\eps\to 0} \left(\frac{\eps^{1/4}}{|\log h_0|}\int_{h_0}^1 h^{-\alpha \theta_0/2}\epsilon_F^{}(h)\frac{\dh}{h} + \frac{\eps}{{h_0}^2|\log h_0|}\right) = 0.
$$
It allows us to deal with the extra term in the right hand side of \eqref{Kfinal}. The other terms are the same as the ones in \eqref{eqII}, and so, can be treated in the same way. 
Thus, choosing $\eta=|\log h_0|^{-1/4}$, from Lemma \ref{lem:compact}, we conclude as before that the sequence $\{\rho_\eps\}_{\eps>0}$ is compact in $L^2_{loc}([0,T]\times\RR^d)$.
\end{proof}

\end{document}